\documentclass[a4paper]{article}
\usepackage[utf8]{inputenc}

\usepackage[pages=all, color=black, position={current page.south}, placement=bottom, scale=1, opacity=1, vshift=5mm]{background}
\SetBgContents{
}      %

\usepackage[margin=1in]{geometry} %
\usepackage{amsmath}
\usepackage{comment}
\usepackage{float}
\usepackage{stmaryrd}
\usepackage{authblk}
\usepackage{amsthm}
\usepackage{amssymb}
\usepackage[linesnumbered,ruled]{algorithm2e}
\usepackage{graphicx}
\usepackage{pgfplots}
\usepackage{pgfplotstable}
\usepgfplotslibrary{colorbrewer}
\pgfplotsset{
        cycle list/Dark2-7,
        cycle multiindex* list={
            mark list*\nextlist
            Dark2-7\nextlist
        },
    }
\pgfplotsset{every axis plot/.append style={thick}}
\usepgfplotslibrary{groupplots,external}
\usetikzlibrary{er,positioning,external}
\tikzset{external/only named=true}
\usepackage{multicol}
\usepackage{mathtools}
\usepackage{caption} %
\usepackage{subcaption} %

\usepackage{hyperref}
\newtheorem{theorem}{Theorem}

\usepackage[sort&compress,numbers,square]{natbib}
\usepackage{graphicx}
\graphicspath{{fig/}}
\usepackage{mathrsfs} %

\newcommand \bfn{\boldsymbol{n}}
\newcommand\cellO{T}
\newcommand\faceO{F}
\newcommand\cellg{K}
\newcommand\TcapK{\tilde{E}}

\pgfplotsset{compat=1.17}

\begin{document}

\title{A comparison of non-matching
techniques for the finite element approximation of interface problems}

\author[1,2,b]{Daniele Boffi}
\author[3,b]{Andrea Cangiani}
\author[3,a]{Marco Feder}
\author[4,5,b]{\authorcr Lucia Gastaldi}
\author[3,b]{Luca Heltai}
\affil[1]{CEMSE division, King Abdullah University of Science and Technology (KAUST), Thuwal, Saudi Arabia}
\affil[2]{
Dipartimento di Matematica ‘F. Casorati’, Universit\`a degli Studi di Pavia, Pavia, Italy}
\affil[3]{Mathematics Area, SISSA,
Trieste, Italy}
\affil[4]{DICATAM Sez. di Matematica, Universit\`a degli Studi di Brescia, Brescia, Italy}
\affil[5]{Istituto di Matematica Applicata e Tecnologie Informatiche (IMATI-CNR), Pavia, Italy\vspace{1cm}}

\affil[a]{Lead author, responsible for the overall direction and content of the manuscript}
\affil[b]{Contributing author, provided significant input and contributions to the manuscript}

\maketitle

{\centering\small Dedicated to Leszek Demkowicz on the occasion of his 70\textsuperscript{th} birthday.\par}

\begin{abstract}
We perform a systematic comparison of various numerical schemes for the approximation of interface problems. We consider unfitted approaches in view of their application to possibly moving configurations.
Particular attention is paid to the implementation aspects and to the analysis of the costs related to the different phases of the simulations.

\end{abstract}

\section{Introduction}
\label{sec:intro}
The efficient numerical solution of partial differential equations modeling the interaction of physical phenomena across interfaces with complex, possibly moving, shapes is of great importance in many scientific fields. We refer, for instance, to %
fluid-structure interaction, %
or crack propagation, just to mention two relevant examples.
A crucial issue is the handling of computational grids. In this respect, we can classify computational methods for interface problems into two families: boundary fitted methods and boundary unfitted methods. For time dependent problems, the former are typically handled using the Arbitrary Eulerian Lagrangian formulation (\cite{Donea_ALE}, \cite{Hirt1997AnAL}), where meshes are deformed in a conforming way with respect to movements of the physical domains. In this case, the imposition of interface conditions is usually easy to implement. However, an accurate description of both meshes is required, and the allowed movements are restricted by the topological structure of the initial state. When topology may change, or when the grid undorgoes severe deformations, these methods require remeshing. An operation which is computationally very expensive in time dependent scenarios and three dimensional settings. Conversely, unfitted approaches are based on describing the physical domains as embedded into a constant background 
mesh. As it does not require remeshing, this approach is extremely flexible, but it requires sophisticated methods to represent interfaces. Among unfitted approaches, we mention the Immersed Boundary Method~\cite{peskin_2002,Boffi_Gastaldi_Heltai_Peskin},
the Cut Finite Element Method~\cite{cutFEM}, and the Extended Finite Element Method~\cite{ExtendedFEM}. Another popular choice is the Fictitious Domain Method with Lagrange multipliers, proposed by Glowinski, Pan and Periaux in~\cite{GlowinskiPanPeriaux1994} for a Dirichlet problem, analyzed in~\cite{GiraultGlowinski}, and then extended to particulate flows in~\cite{GLOWINSKI1999755}. 

The use of a Lagrange multiplier for dealing with Dirichlet boundary conditions was introduced in the seminal work by Babu\v{s}ka \cite{Babuska}.
This is formulated in terms of a symmetric saddle point problem where the condition at the interface is enforced through the use of a Lagrange multiplier. The main drawback of this method is that it suffers from a loss of accuracy at interfaces, even if it is %
known~\cite{Heltai2019} that this detrimental effect on the convergence properties of the approximate solution is a local phenomenon, restricted to a small neighbourhood of the interface. From the computational standpoint, the Fictitious Domain Method  poses the additional challenge of assembling coupling terms involving basis functions living on different meshes. In this context, one can distribute quadrature points on the immersed mesh and let it drive the integration process, or %
one may compute a composite quadrature rule by identifying exact intersections (polytopes, in general) between the two meshes. This approach has been presented in different papers and frameworks: for instance, in~\cite{parmoonolith}, a high performance library has been developed to perform such tasks, and in~\cite{Boffi_Credali_Gastaldi} it has been shown how composite rules on interfaces turn out to be necessary to recover optimal rates for a fluid-structure interaction problem where both the solid and the fluid meshes are two dimensional objects. In the cut-FEM framework, we mention ~\cite{cut3D} for an efficient implementation of Nitsche's method on three dimensional overlapping meshes.

In this paper we consider  a two or three dimensional domain with an immersed interface of co-dimension one and we study the numerical approximation of the solution of an elliptic PDE whose solution is prescribed along the interface. Although in our case the domain and the interface are fixed, we are discussing the presented numerical schemes in view of their application to more general settings.
We perform a systematic comparison between three different unfitted approaches, analyzing them in terms of accuracy, computational cost, and implementation effort. %
In particular, we perform a comparative analysis in terms of accuracy and CPU times for the Lagrange multiplier method, Nitsche's penalization method, and cut-FEM for different test cases, and discuss the benefit of computing accurate quadrature rules on mesh intersections.

The outline of the paper is as follows. In Section~\ref{sec:model-problem} we introduce the model Poisson problem posed on a domain with an internal boundary and in Section~\ref{sec:methods} we review the Lagrange multiplier method, the Nitsche
penalization method, and the cut-FEM method for its solution. %
In Section \ref{sec:integration} we
discuss the central issue of the numerical
integration of the coupling terms, while in
Section \ref{sec:numberical_experiments} we present a numerical comparison of
the three methods, focusing on the validation of the implementation, on how
coupling terms affect the accuracy of numerical solutions, and on computational times.
Finally, the results are summarised in Section~\ref{sec:conc}.

\section{Model problem and notation}
\label{sec:model-problem}
Let $\omega$ be a closed and bounded domain of $\mathbb{R}^d$, $d=2,3$, with Lipschitz continuous boundary $\gamma \coloneqq \partial \omega$, and $\Omega \subset \mathbb{R}^d$ a Lipschitz domain such that $\omega \Subset \Omega$; see Figure~\ref{fig:domains} for a prototypical configuration. We consider the model problem
\begin{equation}
\label{problem}
    \begin{cases}
  -\Delta u &= f  \qquad \text{in}\ \Omega \setminus \gamma, \\
  u &= g \qquad \text{on}\ \gamma, \\
  u &= 0 \qquad \text{on}\ \Gamma \coloneqq \partial \Omega,
\end{cases}
\end{equation}
for given data $f \in L^2(\Omega)$ and $g \in H^{\frac{1}{2}}(\gamma)$.
 Throughout this work we refer to $\Omega$ as the \textit{background} domain, while we refer to $\omega$ as the \textit{immersed} domain, and $\gamma$ as the \emph{immersed boundary}. The rationale behind this setting is that it allows to solve problems in a complex and possibly time dependent domain $\omega$, by embedding the problem in a simpler background domain $\Omega$ -- typically a box -- and imposing some constraints on the immersed boundary $\gamma$. For the sake of simplicity, we consider the case in which the immersed domain is \textit{entirely} contained in the background domain, but more general configurations may be considered. 
 
 As ambient spaces for~\eqref{problem}, we consider $V(\Omega) \coloneqq H_0^1(\Omega)= \{ v \in H^1(\Omega): v_{|\Gamma} =0 \}$ and $Q(\gamma) \coloneqq H^{-\frac{1}{2}}(\gamma)$.  Given a domain $D\subset\mathbb{R}^d$ and  a real number $s\ge 0$, we denote by $\|\cdot\|_{s,D}$ the standard norm of $H^s(D)$. In particular, $\|\cdot\|_{0,D}$ stands for the $L^2$-norm stemming from the standard $L^2$-inner product  $(\cdot, \cdot)_{D}$ on $D$. Finally, with $\langle \cdot, \cdot \rangle_{\gamma}$ we denote the standard duality pairing between $Q(\gamma)$ and its dual $Q'(\gamma)=H^{\frac12}(\gamma)$. 
 
 Across the immersed boundary $\gamma$, we define the \textit{jump} operator as
\begin{align*}
      \llbracket     v \rrbracket_{|\gamma} = v^{+} - v^{-}, \\
    \llbracket \boldsymbol{\tau} \rrbracket_{|\gamma} = \boldsymbol{\tau^+} - \boldsymbol{\tau}^-,
\end{align*}
for smooth enough scalar- and vector- valued functions $v$ and $\boldsymbol{\tau}$. Here, $v^{\pm}$ and $\boldsymbol{\tau}^{\pm}$ are external and internal traces defined according to the direction of the outward normal $\boldsymbol{n}$ to $\omega$ at $\gamma$.

Problem \eqref{problem} can be written as a constrained minimization problem by introducing the Lagrangian $\mathcal{L} : V(\Omega) \times Q(\gamma)  \rightarrow \mathbb{R}$ defined as \begin{equation}
    \mathcal{L}(v,q)\coloneqq\frac{1}{2}(\nabla v,\nabla v)_{\Omega} - (f,v )_{\Omega} + \langle q, v-g \rangle_{\gamma}.
\end{equation}
Looking for stationary points of $\mathcal{L}$ gives the following saddle point problem of finding a pair $(u,\lambda) \in V(\Omega) \times Q(\gamma)$ such that

\begin{eqnarray}
\label{eqn:LM1}
(\nabla u, \nabla v)_{\Omega} + \langle\lambda, v\rangle_{\gamma} &=& (f,v)_{\Omega} \qquad \forall v \in V(\Omega), \\
\label{eqn:LM2}
\langle  q, u\rangle_{\gamma} &=& \langle q,g\rangle_{\gamma} \qquad \forall q \in Q(\gamma),
\end{eqnarray}
Below in Theorem~\ref{th:existence_LM} we show that this problem admits a unique solution.
Starting from \eqref{problem} and integrating by parts, one can easily show that
setting $\lambda =- \llbracket {\nabla u \cdot \boldsymbol{n}}\rrbracket_{|\gamma}$,
the pair $(u,\lambda)\in V(\Omega) \times Q(\gamma)$ is the solution of~\eqref{eqn:LM1}-\eqref{eqn:LM2}. Conversely, with proper choices for $v\in V(\Omega)$ in~\eqref{eqn:LM1} one gets that $-\Delta u=f$ in $\Omega\setminus\gamma$, while~\eqref{eqn:LM2} implies $u=g$ on $\gamma$, so that~\eqref{eqn:LM1}-\eqref{eqn:LM2} are equivalent to~\eqref{problem} with  $-\lambda$ equal to the jump of the normal derivative of $u$ on the interface $\gamma$. Moreover, if the datum $g$ is sufficiently smooth, say $g \in H^s(\gamma)$ for $s>1$, we can further take $\lambda \in L^2(\gamma)$ and use in practice $Q(\gamma) = L^2(\gamma)$. 

In the following theorem we sketch the proof of existence and uniqueness of the solution of~\eqref{eqn:LM1}-\eqref{eqn:LM2}.
\begin{theorem}
Given $f\in L^2(\Omega)$ and $g\in H^{1/2}(\gamma)$, there exists a unique solution of
Problem~\eqref{eqn:LM1}-\eqref{eqn:LM2} satisfying the following stability estimate
\[
\|u\|_{1,\Omega}+\|\lambda\|_{-1/2,\gamma}\le
C(\|f\|_{0,\Omega}+\|g\|_{1/2,\gamma}).
\]
\label{th:existence_LM}
\end{theorem}
\begin{proof}
Problem~\eqref{eqn:LM1}-\eqref{eqn:LM2} is a saddle point problem, hence to show existence and uniqueness of the solution we need to check the \emph{ellipticity on the kernel} and the \emph{inf-sup condition}~\cite{bbf}. The kernel
$\mathbb{K}=\{v\in V(\Omega): \langle q,v\rangle=0\ \forall q\in Q(\gamma)\}$,
can be identified with the subset of functions in $V(\Omega)$ with vanishing trace along $\gamma$. Then thanks to the Poincaré inequality we have that there exists $\alpha_0>0$ such that
\[
(\nabla u, \nabla u)_{\Omega}\ge \alpha_0 \|u\|^2_{1,\Omega}.
\]
The inf-sup condition can be verified using the definition of the norm in $Q(\gamma)$ and the fact that, by the trace theorem, for each $w\in H^{1/2}(\gamma)$ there exists at least an element $v\in V(\Omega)$
such that $v=w$ on $\gamma$, with  $\|v\|_{1,\Omega}\le C_1\|w\|_{1/2,\gamma}$. 
Hence we get the inf-sup condition
\[
\inf_{q\in Q(\gamma)}\sup_{v\in V(\Omega)}
\frac{\langle q,v\rangle}{\|q\|_{-1/2,\gamma}\|v\|_{1,\Omega}}\ge\beta_0,
\]
with $\beta_0=1/C_1$.
\end{proof} 
A detailed analysis for the Dirichlet problem, where the Lagrange multiplier is used to impose the boundary condition, can be found in the pioneering work by Glowinski, Pan and Periaux~\cite{GlowinskiPanPeriaux1994}.

\section{Non-matching discretizations}
\label{sec:methods}
We assume that both $\Omega$  and  $\omega$ are Lipschitz domains and we
discretize the problem introducing computational meshes for the domain $\Omega$ and for
the immersed boundary $\gamma$ which are  
\textit{unfitted} with respect to each other in that they are constructed independently.
The computational meshes $\Omega_h$ of $\Omega$ and $\gamma_h$ of $\gamma$
consist of disjoint elements such that $\Omega = \bigcup_{\cellO \in \Omega_h}
\cellO$ and $\gamma = \bigcup_{\cellg \in \gamma_h} \cellg$. When $d=2$,
$\Omega_h$ will be a triangular or quadrilateral mesh and $\gamma_h$ a mesh
composed by straight line segments. For $d=3$, $\Omega_h$ will be a tetrahedral
or hexahedral mesh and $\gamma_h$ a surface mesh whose elements are triangles or
planar quadrilaterals embedded in the three dimensional space. We denote by
$h_\Omega$ and $h_\gamma$ the mesh sizes of $\Omega_h$ and $\gamma_h$,
respectively. For simplicity, we ignore geometrical errors in the
discretizations of $\Omega$, and $\gamma$, and we assume that the mesh sizes
$h_\Omega$ and $h_\gamma$ are small enough so that the geometrical error is
negligible with respect to the discretization error
(see Sect.~\ref{sec:level_set_splitting} for a quantitative estimate of the
geometrical error in our numerical experiments).

We consider discretizations of $V(\Omega)$ based on standard Lagrange finite elements, namely
\begin{align}
\label{V_h}
V_h(\Omega) \coloneqq \{v \in H^1_{0}(\Omega): v|_{\cellO}  \in \mathcal{R}^p(\cellO), \forall\cellO \in \Omega_h \} \qquad p \geq 1,
\end{align}
with $\mathcal{R}^p(\cellO):=\mathcal{P}^p(\cellO)$ or $\mathcal{Q}^p(\cellO)$, the spaces of polynomials of total degree up to $p$
or of degree $p$ separately in each variable, respectively, depending on whether  $\cellO$ is a simplex or a quadrilateral/hexahedron. 
For the cut-FEM method, the corresponding spaces will be constructed separately in each subdomain; this results in a doubling of the degrees of freedom on the elements cut by the interface (see Sect.~\ref{subsec:cut-FEM}).

For the Lagrange multiplier space $Q(\gamma)$ we employ the space of piecewise-polynomial functions
\begin{align}
Q_h(\gamma) \coloneqq \{q \in L^2(\gamma): q|_{\cellg} \in \mathcal{R}^p(\cellg), \forall\cellg \in \gamma_h\} \qquad p \geq 0,
\end{align} 
which we equip with the mesh dependent norm 
\begin{equation}
    \label{eqn:frac_norm}
    \| q \|_{-\frac{1}{2},\gamma} \coloneqq \|h^{-\frac{1}{2}} q \|_{0,\gamma}\qquad\forall q\in Q_h(\gamma),
\end{equation}
where $h$ is the piecewise constant function given by  $h|_{\cellg}=h_\cellg$ the diameter of $\cellg$ for each $\cellg\in\gamma_h$. The definition and the notation are justified by the fact that, on a quasi-uniform meshe, the mesh dependent norm is equivalent to the norm in $H^{1/2}(\gamma)$.

\subsection{The method of Lagrange Multipliers}
\label{subsubsec:LM_discretisation}
The discrete counterpart of \eqref{eqn:LM1}-\eqref{eqn:LM2} is to find a pair $(u_h, \lambda_h) \in V_h \times Q_h$ such that
\begin{eqnarray}
\label{eqn:LM_algebraic_1}
(\nabla u_h, \nabla v_h)_{\Omega} + \langle \lambda_h, v_h\rangle_{\gamma} &=& ( f,v_h)_{\Omega} \qquad \forall v_h \in V_h, \\
\label{eqn:LM_algebraic_2}
\langle  q_h,u_h\rangle_{\gamma} &=& \langle q_h, g\rangle_{\gamma} \qquad \forall q_h \in Q_h.
\end{eqnarray}
The next theorem states existence, uniqueness, and stability of the discrete solution together with optimal error estimates.
\begin{theorem}
Assume that the mesh $\gamma_h$ is quasi uniform and that there exists a positive constant ${\rm C}_r$ independent of $h_\Omega$ and $h_\gamma$ such that $h_\Omega/h_\gamma\le {\rm C}_r$. Then, there exists a unique solution $(u_h,\lambda_h)\in V_h\times Q_h$ of Problem~\eqref{eqn:LM_algebraic_1}-\eqref{eqn:LM_algebraic_2}.
Moreover, it holds
\begin{equation}
\|u-u_h\|_{1,\Omega}+\|\lambda-\lambda_h\|_{-1/2,\gamma}\le
{\rm C} \inf_{\substack{v_h\in V_h\\ \mu_h\in Q_h}}
\left(\|u-v_h\|_{1,\Omega}+\|\lambda-\mu_h\|_{-1/2,\gamma}\right),
\label{eq:err_estimate}
\end{equation}
 with ${\rm C}>0$ a constant independent of the mesh sizes $h_\Omega$ and $h_\gamma$.
\label{th:existenceLMh}
\end{theorem}
\begin{proof}
The existence, uniqueness, and stability of the discrete solution can be obtained by showing that there exist positive constants $\alpha$ and $\beta$, independent of $h_\Omega$ and $h_\gamma$, such that the ellipticity on the kernel and inf-sup condition hold true at the discrete level~\cite{bbf}.
Since $Q_h\subset L^2(\gamma)$, we have that the discrete kernel $\mathcal{K}_h=\{v_h\in V_h: \langle q_h,v_h\rangle=0\, \forall q_h\in Q_h\}$ contains element with $\int_\gamma v_h ds=0$. Hence, the Poincar\'e inequality (see~\cite[(5.3.3)]{BrennerScott})
\[
\|v_h\|_{1,\Omega}\le {\rm C}_\Omega\left(\left|\int_\gamma v_h ds\right|+\|\nabla v_h\|_{0,\Omega}\right)=
{\rm C}_\Omega\|\nabla v_h\|_{0,\Omega},
\]
implies that
\[
(\nabla v_h,\nabla v_h)_\Omega\ge \alpha \|v_h\|^2_{1,\Omega}\quad \forall v_h\in\mathcal{K}_h. 
\]
The discrete inf-sup condition 
\[
\inf_{q_h\in Q_h}\sup_{v_h\in V_h}
\frac{\langle q_h,v_h\rangle}{\|q_h\|_{-1/2,\gamma}\|v_h\|_{1,\Omega}}\ge\beta,
\]
is more involved and makes use of the continuous inf-sup, together with Cl\'ement interpolation, trace theorem and inverse inequality. The interested reader can find the main arguments of this proof in~\cite[sect. 5]{BG_numermath}.

Thanks to the above conditions, the theory on the approximation of saddle point problems gives both existence and uniqueness of the solution of Problem~\eqref{eqn:LM_algebraic_1}-\eqref{eqn:LM_algebraic_2} satisfying the a priori estimate
\[
\|u_h\|_{1,\Omega}+\|\lambda_h\|_{-1/2,\gamma}\le
{\rm C}(\|f\|_{0,\Omega}+\|g\|_{1/2,\gamma}),
\]
and the error estimate~\eqref{eq:err_estimate}.
\end{proof}

Given basis functions  $ \{ v_i\}_{i=1}^N$ and $\{ q_i\}_{i=1}^M$ such that $V_h \coloneqq \text{span} \{ v_i\}_{i=1}^N$ and $Q_h \coloneqq \text{span} \{ q_i\}_{i=1}^M$, we have that \eqref{eqn:LM_algebraic_1}, \eqref{eqn:LM_algebraic_2} can be written as the following algebraic problem
\begin{eqnarray}
\label{eqn:LM_linear_system}
  \left(\begin{array}{cc}
    A & C^\top \\ C & 0
  \end{array}\right)
  \left(\begin{array}{cc}
    U \\ \lambda
  \end{array}\right)
  =
  \left(\begin{array}{cc}
    F \\ G
  \end{array}\right)
\end{eqnarray}
where
\begin{eqnarray*}
A_{ij} &=& (\nabla v_j, \nabla v_i)_{\Omega}   \qquad i,j=1,\dots,N \\
C_{\alpha j} &=& \langle q_\alpha,v_j \rangle_{\gamma}  \qquad j=1,\dots,N, \alpha = 1,\dots, M \\
F_{i} &=& (f, v_i)_{\Omega}   \qquad i=1,\dots,N \\
G_{\alpha} &=& \langle  q_\alpha, g \rangle_{\gamma} \qquad \alpha = 1,\dots, M.
\end{eqnarray*}

To solve the block linear system \eqref{eqn:LM_linear_system} we use Krylov subspace iterative methods applied to the Schur complement system:
\begin{align}
    \lambda = S^{-1}(C A^{-1} F-G), \\
    U = A^{-1} (F - C^\top \lambda),
\end{align} where $S \coloneqq C A^{-1} C^\top$, and we use  $CAC^\top + M$ as preconditioner for $S$, where $M$ is the immersed boundary mass matrix with entries $(M)_{ij} = \langle q_j, q_i \rangle_{\gamma}$.

\begin{figure}[ht]
    \centering
    \includegraphics{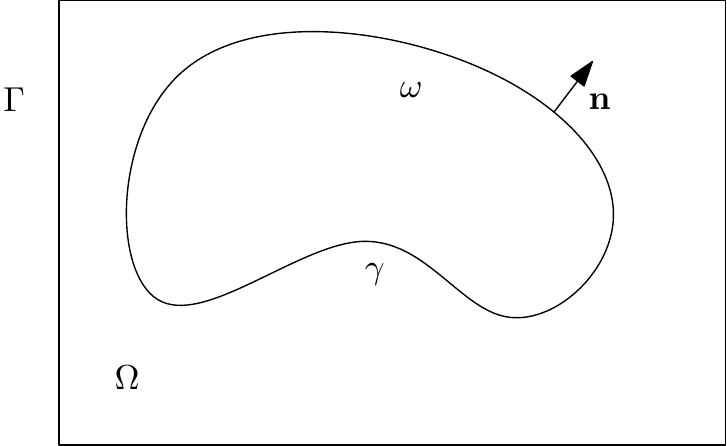}
    \caption{Model problem setting, with immersed domain $\omega$, immersed boundary $\gamma$, and background domain $\Omega$.}
    \label{fig:domains}
\end{figure}

We next show how the Lagrange multiplier formulation is directly linked to a penalization approach used to impose the Dirichlet condition $u=g$ on the internal curve $\gamma$, by locally eliminating the multiplier in the same spirit of the work by Stenberg~\cite{Stenberg}.

\subsection{The method of Nitsche}

Instead of enforcing the constraint on $\gamma$ with a multiplier, it is possible to impose it weakly through a penalization approach following the so-called method of Nitsche. Here we show that the enforcement of boundary conditions via Nitsche's method can be derived from a \emph{stabilized} Lagrange multiplier method by adding a consistent term that penalizes the distance between the discrete multiplier $\lambda_h$ and the normal derivative~\cite{Stenberg}. With this in mind, we penalize the \textit{jump} of the normal derivative along the internal curve $\gamma$ to impose the constraint $u=g$. The consistency here follows from the observation that at the continuous level the multiplier is the jump of the normal derivative on the interface. 

We define $h(\boldsymbol{x})$ as the piecewise constant function describing the mesh-size of $\gamma$ and we choose the discrete spaces $V_h$ and $Q_h$ as in the Lagrange multiplier case.

Adding the normal gradient penalization term to the Lagrange multiplier formulation~\eqref{eqn:LM_algebraic_1}-\eqref{eqn:LM_algebraic_2} leads to the problem of seeking a pair $(u_h,\lambda_h) \in V_h \times Q_h$ such that
\begin{eqnarray*}
\label{eqn:LMStabilized}
(\nabla u_h, \nabla v_h)_{\Omega} + \langle \lambda_h, v_h\rangle_{\gamma} 
+\frac1{\beta}\langle \llbracket  \nabla v_h\cdot\bfn \rrbracket ,h (\lambda_h + \llbracket  \nabla u_h\cdot\bfn \rrbracket ) \rangle_{\gamma}
&=& ( f,v_h)_{\Omega} \qquad \forall v_h \in V_h, \\
\label{eqn:LMStabilized2}
\langle q_h,u_h \rangle_{\gamma} - \frac{1}{\beta} \langle q_h , h (\lambda_h + \llbracket \nabla u_h\cdot\bfn \rrbracket ) \rangle_{\gamma} &=& \langle q_h,g\rangle_{\gamma} \qquad \forall q_h \in Q_h,
\end{eqnarray*} 
where $\beta$ is a positive penalty parameter. In the discrete setting, $\langle q, v\rangle_\gamma$ is identified with the scalar product in $L^2(\gamma)$ for $q\in Q_h$ and $v\in V_h$
and we use the notation $\langle q,v\rangle_{\gamma}=\sum_{\cellg\in\gamma_h} (q,v)_\cellg.$
The second equation gives 
\[ 
\langle q_h,  u_h - \frac{1}{\beta}  h (\lambda_h + \llbracket  \nabla u_h\cdot\bfn \rrbracket  ) -g \rangle_{\gamma} =0 
\qquad \forall q_h \in Q_h,
\]
and, introducing the $L^2$-projection on $Q_h$ as $\Pi_h: L^2(\gamma) \rightarrow Q_h$, we can eliminate the multiplier locally on each element $\cellg \in \gamma_h$:
\[\lambda_h|_\cellg = -\left(\Pi_h \llbracket \nabla u_h\cdot\bfn \rrbracket \right)_{\cellg} + \beta h_\cellg^{-1} \left(\Pi_h(u_h -g)\right)_{\cellg}.\]

We observe that it is possible to formally refine $\gamma_h$ to $\gamma'_h$ such that the element boundaries of $\Omega_h$ coincide with some element boundaries of the immersed grid $\gamma_h'$. If we now choose a space $Q'_h$ which contains piece wise polynomials of degree compatible with that of the elements in $V_h$, we can avoid the projection operator altogether, and are allowed to write
\[
\lambda'_h= -\llbracket  \nabla u_h\cdot\bfn \rrbracket  + \beta h^{-1} (u_h -g).
\]
Inserting this back into the first equation we get 
the variational problem (in which no multiplier is involved) of finding $u_h \in V_h$ such that 
\begin{equation}
\label{eqn:penalisation}
\aligned
    (\nabla u_h, \nabla v_h)_{\Omega} &- \langle \llbracket \nabla u_h\cdot\bfn \rrbracket , v_h \rangle_{\gamma} - \langle \llbracket \nabla v_h\cdot\bfn \rrbracket , u_h \rangle_{\gamma} + \beta  \langle h^{-1}u_h,v_h \rangle_{\gamma}\\ 
    &= (f,v_h)_{\Omega} - \langle \llbracket \nabla v_h\cdot\bfn \rrbracket ,g \rangle_{\gamma} + \beta  \langle h^{-1}g,v_h \rangle_{\gamma},
\endaligned
\end{equation} 
holds for every $v_h \in V_h$. 

Equation~\eqref{eqn:penalisation} represents the Nitsche method~\cite{Nitsche} applied to Problem~\eqref{problem}. Owing to the non-matching nature of our discretization, the immersed and background mesh facets are not expected to be  aligned in general. If indeed for all facets (elements) $\cellg$ of $\gamma$ and facets $\faceO$ of $\Omega_h$ we have $\mathcal{H}^{d-1}(\cellg \cap \faceO)=0$, with $\mathcal{H}^{d-1}$ denoting the Hausdorff measure in $d-1$-dimension, then 
the variational problem can be simplified since all the jump terms vanish.

\subsection{The cut-FEM method}
\label{subsec:cut-FEM}

When using the cut-FEM discretization approach, one changes the perspective of the original variational problem, which is no longer solved on a single space defined globally on $\Omega$: instead, one solves two separate problems on the two domains $\omega$ and $\Omega\setminus \omega$. This approach gives additional flexibility on the type of problems that can be solved. For example, problems with more general transmission conditions across $\gamma$ where the solution $u$ is allowed to jump. Moreover, separating the problem transforms the cut-FEM method into a boundary-fitted approach, where the approximation space is changed to resolve the interface.

The usual approach to impose constraints on $\gamma$ in the cut-FEM method is to use Nitsche's method applied on the two subdomains separately:
\[
\begin{cases}
\Omega^1 \coloneqq \omega, \\
\Omega^2 \coloneqq \Omega \setminus \Bar{\omega}.
\end{cases}
\]
In this context, it is necessary to take special care of those elements of $\Omega_h$ that are cut by $\gamma$. First, we introduce the corresponding computational meshes $\Omega_h^i$ given by
\[
\Omega_h^i \coloneqq \{\cellO \in \Omega_h: \cellO \cap \Omega_i \ne \emptyset  \} \qquad i=1,2,
\] 
and notice that both meshes share the set of cells intersected by the curve $\gamma$, namely
\[
\tau \coloneqq \{\cellO \in \Omega_h: \cellO \cap \gamma \ne \emptyset \}, \qquad \tau^i \coloneqq \{ \tilde{\cellO}^i\coloneqq \cellO\cap\Omega^i, \ \cellO \in \tau \},
\]
where we distinguish between entire cells that intersect $\gamma$ (these are in the set $\tau$) and cut cells (with arbitrary polytopal shape, which are in the set $\tau^i$), i.e., $(\cellO\in \tau) = (\tilde{\cellO}^1\in\tau^1)\cup(\tilde{\cellO}^2\in\tau^2)\cup (\gamma\cap \cellO)$ with $\tilde{\cellO}^1\cap\tilde{\cellO}^2 = \emptyset$.

When defining a finite element space on these elements, one uses the same definition of the original finite element space defined on entire elements $\cellO \in \tau$, which is then duplicated and restricted to the corresponding domain, i.e., one introduces on $\Omega_h^i$, $i=1,2$ the discrete spaces
\[
V_h^{i} \coloneqq V_h(\Omega_h^i)|_{\Omega^i}=\{v_h|_{\Omega^i},\ v_h\in V_h(\Omega_h^i)\} .
\]
Then, applying twice Nitsche's method requires to find $u_h^i \in V_h^i$ such that $$a_h^i(u_h^i,v_h^i) = l_h^i(v_h^i) \qquad \forall v_h^i \in V_h^i \qquad i=1,2, $$ with
\begin{equation}
    a_h^i(u_h^i,v_h^i) = (\nabla u_h^i, \nabla v_h^i)_{\Omega_h^i\cap \Omega^i} - \langle \llbracket\nabla u_h^i\cdot\bfn\rrbracket, v_h^i\rangle_{\gamma} - \langle u_h^i, \llbracket\nabla v_h^i\cdot\bfn\rrbracket \rangle_{\gamma} + \frac{\beta_1}{h} \langle  u_h^i, v_h^i\rangle_{\gamma} ,
\end{equation}
and 
\begin{equation}
    l_h^i(v_h^i) = (f,v_h^i)_{\Omega_h^i\cap \Omega^i} + \left\langle g,\frac{\beta_1}{h}v_h^i - \llbracket\nabla v_h^i\cdot\bfn\rrbracket \right\rangle_{\gamma} .
\end{equation}
This formulation is known to suffer from the so called \textit{small-cut} problem deriving from the fact that the size of the cuts $\cellO \cap \Omega^i$ cannot be controlled and hence can be arbitrarily small. This may result in a loss of coercivity for the bilinear forms $a_h^i$ when the size of a cut cell goes to zero. As shown in \cite{cutFEM}, the formulation can be stabilized by adding to the bilinear form the following penalty term acting on the interior or exterior faces of the intersected cells, depending on the domain $\Omega_h^i$: 
\begin{align}
    \mathcal{G}_h^i \coloneqq \{ \faceO = \overline{\cellO}_{+} \cap \overline{\cellO}_{-}: \cellO_{+} \in \tau, \cellO_{-} \in \Omega_h^i \} \qquad i=1,2,
\end{align}
\[
j_h^i(u,v) \coloneqq \beta_2 \sum_{\faceO \in \mathcal{G}_h^i} \langle h_\faceO \llbracket \nabla u\cdot\bfn \rrbracket ,\llbracket \nabla v\cdot\bfn \rrbracket \rangle_{\faceO},\]
where $\beta_2 $ is a positive penalty parameter and $h_\faceO$ the size of $\faceO \in \mathcal{G}_h^i$. With such definition at hand  we set  
\begin{align}
V_h &\coloneqq V_h^{1} + V_h^{2}, \text{ with elements } v_h=\begin{cases}v^1_h &\text{ in }\Omega^1\\v^2_h &\text{ in }\Omega^2\end{cases} \\
a_h(u_h,v_h) &\coloneqq \sum_{i=1}^2 (a_h^i(u_h^i,v_h^i) + j_h^i(u_h^i,v_h^i)),\\ 
l_h(v_h) &\coloneqq \sum_{i=1}^2 l_h^i(v_h^i), 
\end{align}
and the method reads: find $u_h \in V_h$ such that
\[
a_h(u_h,v_h) = l_h(v_h) \qquad \forall v_h \in V_h.
\] 

\section{Integration of coupling terms}
\label{sec:integration}
In all three methods, some terms need to be  integrated over the non-matching interface $\gamma$. For example, in the Lagrange multiplier method, we need to compute $\langle \lambda_h, v_h \rangle_{\gamma}$ where $v_h \in V_h$ and $\lambda_h \in Q_h$, while in the Nitsche's interface penalization method and in the cut-FEM method, one needs to integrate terms of the kind $\langle \beta h^{-1} u_h, v_h\rangle_\gamma$, where both $u_h$ and $v_h$ belong to $V_h$, but the integral is taken over $\gamma$.

We start by focusing our attention on the term $\langle \lambda_h, v_h
\rangle_{\gamma}$ where $v_h \in V_h$ and $\lambda_h \in Q_h$. This is delicate
to assemble as it is the product of trial and test functions living on
\emph{different} meshes and it encodes the interaction between the two grids.
Let $\cellg \in \gamma_h$ and $F_\cellg$ be the map $F_\cellg: \hat{\cellg}
\rightarrow \cellg$ from the reference immersed cell $ \hat{\cellg}$ to the
physical cell $\cellg$, $JF_{\cellg}(\boldsymbol{x})$ the determinant of its
Jacobian and assume to have a quadrature rule $\{ \hat{\boldsymbol{x}}_q, w_q
\}_{q=1}^{N_q}$ on $\hat{\cellg}$. Since the discrete functions are piecewise polynomials, we can use the scalar product in $L^2(\gamma)$ instead of the duality pairing, then standard finite element assembly reads:

\begin{equation}
\label{eqn:standardfem}
    \langle \lambda_h,v_h \rangle_{\gamma} = \sum_{\cellg \in \gamma_h} \langle \lambda_h, v_h \rangle_{\cellg}.
\end{equation}

In the forthcoming subsections we discuss three strategies to compute this
integral. Identical considerations apply in order to compute the term  $\beta
\langle h^{-1}u_h,v_h \rangle_{\gamma}$ in \ref{eqn:penalisation}.

In general, such integrals are always computed using quadrature formulas. What
changes is the algorithm that is used to compute these formulas, and the
resulting accuracy. Independently on the strategy that is used to compute the
quadrature formulas, all algorithms require the efficient identification of
pairs of potentially overlapping cells, or to identify background cells where
quadrature points may fall. This task can be performed efficiently by using
queries to R-trees of axis-aligned bounding boxes of cells for both the
background and immersed mesh.

An R-tree is a data structure commonly used for spatial indexing of
multi-dimensional data that relies on organizing objects (e.g., points, lines,
polygons, or bounding boxes) in a hierarchical manner based on their spatial
extents, such that objects that are close to each other in space are likely to
be located near each other in the tree.

In an R-tree, each node corresponds to a rectangular region that encloses a group of objects, and the root node encloses all the objects. Each non-leaf node in the tree has a fixed number of child nodes, and each leaf node contains a fixed number of objects. R-trees support efficient spatial queries such as range queries, nearest neighbor queries, and spatial joins by quickly pruning parts of the tree that do not satisfy the query constraints.

In particular, we build two R-tree data structures to hold the bounding boxes of
every cell of both the background and the immersed meshes. Spatial queries are
performed traversing the R-tree structure generated by the
\texttt{Boost.Geometry} library \cite{BoostLibrary}. The construction of an
R-tree with $M$ objects has a computational cost that is proportional to $O(M
\log(M))$, while the cost of a single query is $O(\log(M))$.

\subsection{Integration driven by the immersed mesh}

\label{subsec:mesh_driven_integration}
Applying straightforwardly a given quadrature rule defined over $\gamma$ to equation~\eqref{eqn:standardfem} gives:
\begin{equation}\label{eq:naive}
\langle \lambda_h, v_h\rangle_{\gamma} \approx \sum_{\cellg \in \gamma_h} \sum_{q=0}^{N_q} \lambda_h \bigl( F_\cellg(\hat{\boldsymbol{x}}_q) \bigr) v_h  \bigl( F_\cellg(\hat{\boldsymbol{x}}_q) \bigr) JF_\cellg(\hat{\boldsymbol{x}}_q) w_q,
\end{equation}
with respect to some reference quadrature points $\hat{\boldsymbol{x}}_q$ and
weights $w_q$, letting the immersed domain drive the integration. In this case
the computational complexity stems from the evaluation of the terms
$v_h(F_\cellg(\hat{\boldsymbol{x}}_q))$, since the position within the
background mesh $\Omega_h$ of the quadrature point
$F_\cellg(\hat{\boldsymbol{x}}_q)$ is not known \textit{a-priori} (see
Figure~\ref{fig:Mesh_intersection_DoF_1D_2D}). A possible algorithm for the
evaluation of $v_h(F_\cellg(\hat{\boldsymbol{x}}_q))$ can be summarized as
follows:
\begin{itemize}
    \item Compute the physical point $\boldsymbol{y} = F_\cellg(\hat{\boldsymbol{x}}_q)$;
    \item Find the cell $\cellO\in\Omega_h$ s.t. $\boldsymbol{y} \in \cellO$;
    \item Given the shape function $\hat{v}_h(\boldsymbol{\hat{x}})$ in the reference element $\hat{\cellO}$, compute $ \hat{v_h} \bigl(G_\cellO^{-1}(\boldsymbol{y}) \bigr)$, where  $G_\cellO : \hat{\cellO} \rightarrow \cellO$ denotes the reference map associated to $\cellO \in \Omega_h$ for the background domain. 
\end{itemize}
The first part of the computation takes linear time in the number of cells of
the immersed mesh $\gamma_h$, while the second part requires a computational
cost that scales logarithmically with the number of cells of the background
grid \emph{for each quadrature point}. Finally, the evaluation of the inverse
mapping $G^{-1}_{\cellO}$ requires Newton-like methods for general unstructured
meshes or higher order mappings (i.e., whenever $F_K$ is non-affine). Overall, implementations based on R-tree
traversal will result in a computational cost that scales at least as $O((M +
N)\log(N))$, where $M$ is the total number of quadrature points (proportional to
the number of cells of the immersed mesh), and $N$ is the number of cells of the
background mesh.

\begin{figure}[ht]
    \centering
    \includegraphics{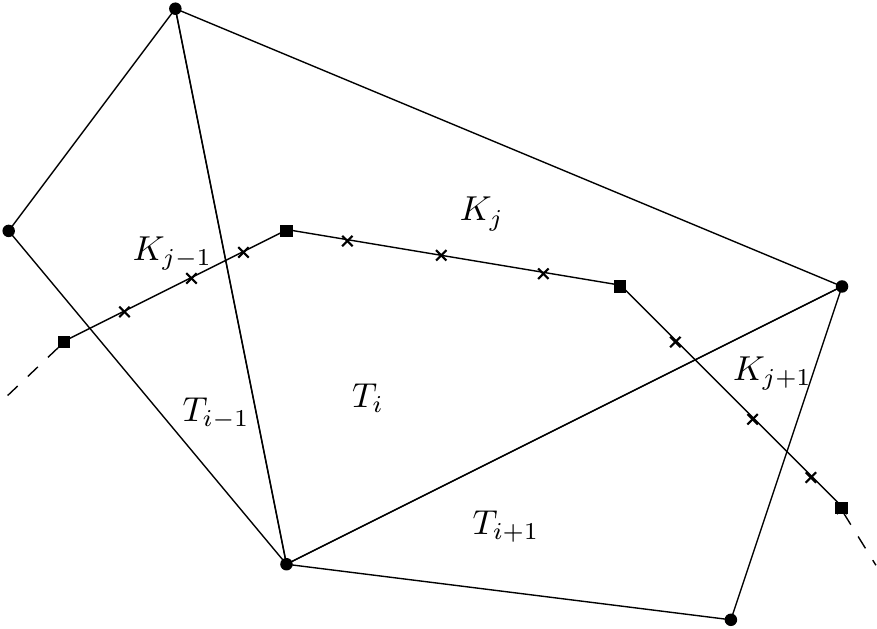}
    \caption{Squares: DoF for linear basis functions attached to some elements $K_j$ of $\gamma$ intersecting a background element $T_i \in \Omega_h$.
    Dots: DoF for a linear Lagrangian basis attached to cells $T_i$. Crosses: quadrature points corresponding to a Gaussian quadrature rule of order 3 on elements of $\gamma$.}
    \label{fig:Mesh_intersection_DoF_1D_2D}
\end{figure}

The main drawbacks of this approach are twofold: i) on one side it does not
yield the required accuracy even if quadrature rules with the appropriate order
are used, due to the piecewise polynomial nature of the integrands, that may
have discontinuities within the element $K$, and ii) since the couplings between
the two grids is based solely on a collection of quadrature points on the
immersed surface, it may happen that two elements overlap, but no quadrature
points fall within the intersection of the two elements. 

This is illustrated in a pathological case in
Figure~\ref{fig:kinks_intersections} (top), where we show a zoom-in of a
solution computed with an insufficient number of quadrature points, and an overly
refined background grid. In this case the quadrature rule behaves like a
collection of Dirac delta distributions, and since the resolution of the
background grid is much finer than the resolution of the immersed grid, one can
recognize in the computed solution the superposition of many small fundamental
solutions, that, in the two-dimensional case, behave like many logarithmic
functions centered at the quadrature points of $\gamma_h$.

\begin{figure}[ht]
    \centering
    \subfloat[\centering Underresolved mesh-driven integration]{{\includegraphics[width=\textwidth]{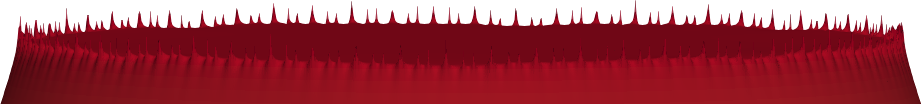} }}
    
    \subfloat[\centering Integration using mesh intersection]{{\includegraphics[width=\textwidth]{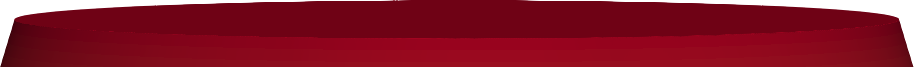} }}%
    \caption{Comparison between exact and non exact integration (zoomed-in on the
    interface in order to highlight the kinks) for a pathological case. (a) Spurious kinks around the interface,
    on the location of the immersed quadrature points. (b) Well-resolved
    solution around the interface.}%
    \label{fig:kinks_intersections}
\end{figure}

These issues with the quadrature driven approach may hinder the convergence
properties of the methods, as shown in \cite{Boffi_Credali_Gastaldi} for a 2D-2D
problem, and require a careful equilibrium between the resolution of the
immersed grid $\gamma_h$, the choice of the quadrature formula on $K$, and the
resolution of the background grid $\Omega_h$. Alternatively, one can follow a
different approach that is based on the identification of the intersections
between the two grids, and on the use of a quadrature rules defined on the
intersection of the two elements, to remove the artifacts discussed above (at
the cost of computing the intersection between two non-matching grids), as shown
in Figure~\ref{fig:kinks_intersections} (bottom).

\subsection{Integration on mesh intersections}
\label{subsec:mesh_intersection_integration}
An accurate computation of the interface terms may be performed by taking into account the intersection between the two grids. First, the non-empty intersections $\TcapK \coloneqq \cellO \cap \cellg \ne \emptyset$ between any $\cellO \in \Omega_h$ and $\cellg \in \gamma_h$ are identified and the intersection is computed accordingly.
Then, given that the restriction of  $v_h$ to $\TcapK$ is smooth, a suitable quadrature formula can be applied in $\TcapK$. Since $\TcapK$ is a polygon in general, we use a sub-tessellation technique, consisting in splitting $\TcapK$ into sub-elements $S\in S_{\TcapK}$ such that $\TcapK=\bigcup_{S\in S_{\TcapK}} S$ and  we use standard Gaussian quadrature rules on each of these sub-elements. See Figures~\ref{fig:Mesh_intersection_1D_2D} and \ref{fig:Mesh_intersection_2D_3D} for an example in 
two and three dimensions, respectively.

In conclusion, interface terms such as $\langle \lambda_h, v_h \rangle_{\gamma}$ are assembled by summing the contribution of each intersection $\TcapK$ computed by appropriate quadrature:
\begin{align}
\label{eqn:composite_quad}
    \langle \lambda_h, v_h \rangle_{\TcapK} = \sum_{S \in S_{\TcapK}} \langle \lambda_h, v_h \rangle_{S} \approx  \sum_{S \in S_{\TcapK}} \sum_{q=1}^{N_q} \lambda_h(F_S(\hat{\boldsymbol{{x}}}_q)) v_h(F_S(\hat{\boldsymbol{x}}_q)) JF_S(\hat{\boldsymbol{x}}_q) w_{q},    
\end{align}
where $F_S \colon \hat{\cellg} \rightarrow S_S$ is the mapping from the reference element $\hat{\cellg}$ to $S$ and $\{ \hat{\boldsymbol{x}}_q, w_q \}_{q=1}^{N_q}$ a suitable quadrature rule defined on $\hat{\cellg}$.

The efficient identification of pairs of potentially overlapping cells is
performed using queries to R-trees of axis-aligned bounding boxes of cells for
both the background and immersed mesh. While this allows to record the indices
of the entries to be allocated during the assembly procedure and the relative
mesh iterators, the actual computation of the \emph{geometric} intersection
$\TcapK$ between two elements $\cellO \subset \Omega_h$ and $\cellg \subset
\gamma$ is computed using the free function \texttt{CGAL::intersection()}. The
resulting polytope is sub-tessellated into simplices, i.e., $\TcapK =
\cup_{i=0}^{N_s} S_i $, even though other techniques for numerical quadrature on
polygons may be employed~\cite{Cangiani_PolyDG}.

In Figures \ref{fig:Mesh_intersection_1D_2D} and \ref{fig:Mesh_intersection_2D_3D} we show this procedure graphically for cells $\cellO$ and $\cellg$ and provide an example for the corresponding sub-tessellations of the intersection $\TcapK$ in two and three dimensions. The resulting sub-tesselations are used to construct the quadrature rules $\{Q_i\}_i$ in \eqref{eqn:composite_quad}.

The overall complexity of the algorithm is $O((N+M) \log(N))$, where $N$ and $M$
are the numbers of cells in the background and immersed mesh, respectively.
Notice, however, that the complexity of the algorithm should be multiplied by
the complexity of the \texttt{CGAL::intersection()} function, which is $O(nm)$
for the intersection of two polygons with $n$ and $m$ vertices, respectively. In
practice, this is not a problem since the number of vertices of the polygons is
typically small, and the complexity of the algorithm is dominated by the
complexity of the R-tree queries, but these costs are non-negligible for large
grids, and many overlapping cells (see, e.g., the discussion
in~\cite{Boffi_Credali_Gastaldi}).

\begin{figure}[ht]
    \centering
    \subfloat[\centering Background cell $T$ and foreground cell $K$.]{{\includegraphics[page=1]{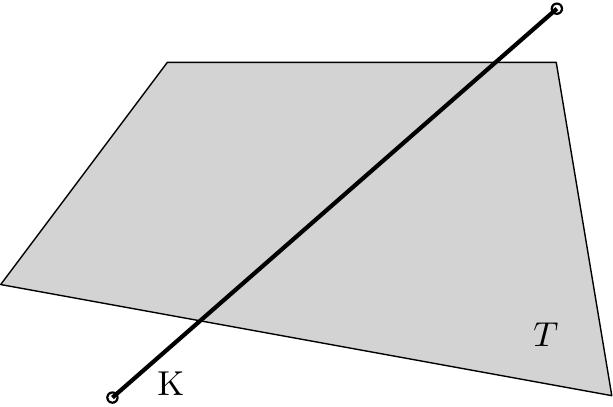} }}%
    \qquad
    \subfloat[\centering The intersection $\TcapK$ is computed and triangulated. In this case, resulting in a single one-dimensional simplex (a segment).]{{\includegraphics[page=2]{fig/intersection_2d.pdf}}}%
    \caption{Triangulation of the intersected region $\TcapK$ for a line immersed in 2D. The one-dimensional cell $K$ is allowed to be positioned arbitrarily w.r.t. to the two-dimensional quadrilateral $T$.}%
    \label{fig:Mesh_intersection_1D_2D}
\end{figure}

\begin{figure}[ht]
    \centering
    \subfloat[\centering In grey the background cell, in green the immersed one.]{{\includegraphics[page=1]{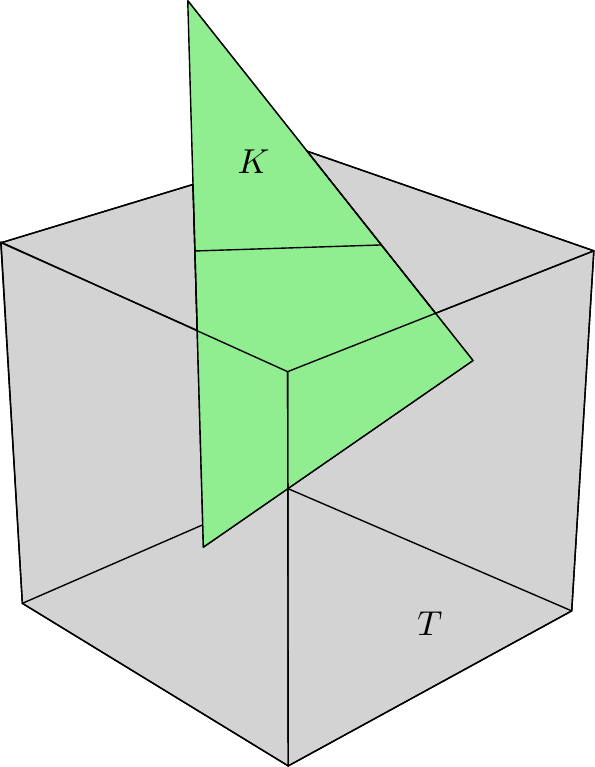} }}%
    \qquad
    \subfloat[\centering The intersection $\TcapK$ is computed and triangulated. ]{{\includegraphics[page=2]{fig/intersection_3d.pdf} }}%
    \caption{Triangulation of the intersected region $\TcapK$ for a triangle cutting a square and the relative sub-tessellation.}%
    \label{fig:Mesh_intersection_2D_3D}
\end{figure}

\subsection{Integration through level set splitting}
\label{sec:level_set_splitting}

If a level set description of the immersed domain is available, this may be used to generate quadrature formulas without explicitly computing the geometric intersection (see \cite{cutFEM} for some implementation details). 
It has also been shown in \cite{Saye} how to generate quadrature rules on different regions of a cut element using $\Psi$, identified as
\[
\aligned
    &O = \{ (x,y) \in \cellO: \Psi >0 \},\\
    &S = \{ (x,y) \in \cellO: \Psi =0 \},\\
    &I = \{ (x,y) \in \cellO: \Psi <0 \},
\endaligned
\]
where $\Psi : \mathbb{R}^d \rightarrow \mathbb{R}$ is the level set function determining the immersed domain. A typical configuration including quadrature points for each entity is shown in Figure~\ref{fig:Quadratures_cutFEM}.

\begin{figure}[ht]
    \centering
    \subfloat[]{{\includegraphics[page=1]{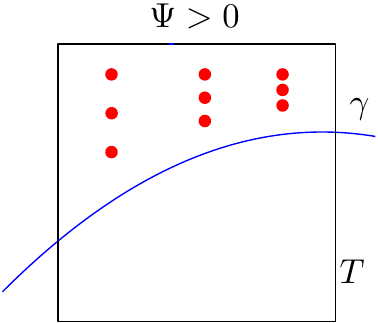}}}%
    \qquad
    \subfloat[]{{\includegraphics[page=2]{fig/quad_psi.pdf}}}%
    \qquad
    \subfloat[]{{\includegraphics[page=3]{fig/quad_psi.pdf}}}%
    \caption{Quadrature points distributed on $O$, $S$ and $I$ for a cell $\cellO$ cut by the curve $\gamma$.}%
    \label{fig:Quadratures_cutFEM}
\end{figure}

In practical computations, however, the interface $\gamma$ is not available in
terms of a simple analytical description of the level set function. Should one
still wish to use quadrature formulas based on a level set, a discrete (possibly
approximated) level set function $\Psi_h$ that is zero on $\gamma_h$ must be
provided when $\gamma_h$ is a triangulated surface. Such level set would also
allow a robust partitioning of the background computational mesh into cells that
are completely inside $\omega$, cells cut by $\gamma$, and cells that are
completely inside $\Omega\setminus\omega$.

We propose a simple implementation of a discrete level set function $\Psi_h$
constructed from a triangulated interface $\gamma_h$. Point classification (i.e.,
detecting if a point is inside or outside $\omega$) is performed using a query
to the \texttt{CGAL} library, to detect if a point is inside or outside the
coarsest simplicial mesh bounded by $\gamma_h$ (denoted by $\mathcal{I}_h$).

We then define a \textit{discrete} level set function $\Psi_h(\boldsymbol{x})$ on top of $\gamma_h$ as follows:
\[
\Psi_h(\boldsymbol{p}) 
=
\begin{cases}
\label{eqn:discrete_signed_distance}
    -d(\boldsymbol{p},\gamma) &  \text{if } \boldsymbol{p} \in \mathcal{I}_h, \\
    d(\boldsymbol{p},\gamma) & \text{if } \boldsymbol{p} \notin \mathcal{I}_h, \\
\end{cases}
\]
where $d(\boldsymbol{p},\gamma) \coloneqq \min_{\boldsymbol{y} \in \gamma} d(\boldsymbol{p},\boldsymbol{y})$ is constructed by first finding the closest elements of $\gamma$ to $\boldsymbol{p}$ using efficient R-tree data structures indexing the cells of the immersed triangulation, and then computing the distance between $\boldsymbol{p}$ and those elements.

This procedure is summarized in Algorithm \ref{alg:discrete_level_set}. We
validate our approach with a manufactured case by choosing randomly distributed
points $\{p_i\}_i$ in the interval $[-1,1]^2$ and computing the relative error
$E_r(\boldsymbol{p}) \coloneqq \frac{|\Psi(\boldsymbol{p}) -
\Psi_h(\boldsymbol{p}) |}{|\Psi(\boldsymbol{p})|}$ for a level set describing a
circle of radius $R=0.3$, and a corresponding approximated grid $\gamma_h$. 

In Figure \ref{fig:Relative_error_discrete_level_set} we report the relative
error committed by replacing the exact level set with the discrete one (induced
by replacing the exact curve with a triangulated one). The initial
discretization of $\gamma$ is chosen so that the error is below $10^{-6}$
everywhere, and we can safely neglect it when computing convergence rates of the error computed w.r.t. exact solutions known on the analytical level set.

As long as such geometrical error is not dominating we can observe optimal rates
in the numerical experiments for cut-FEM. At the same time, this setting allows
for a fair comparison between all the schemes. We expect that other practical
implementations would require similar tasks and hence that the observed
computational cost is representative.

\begin{figure}[ht]
    \centering
    \includegraphics{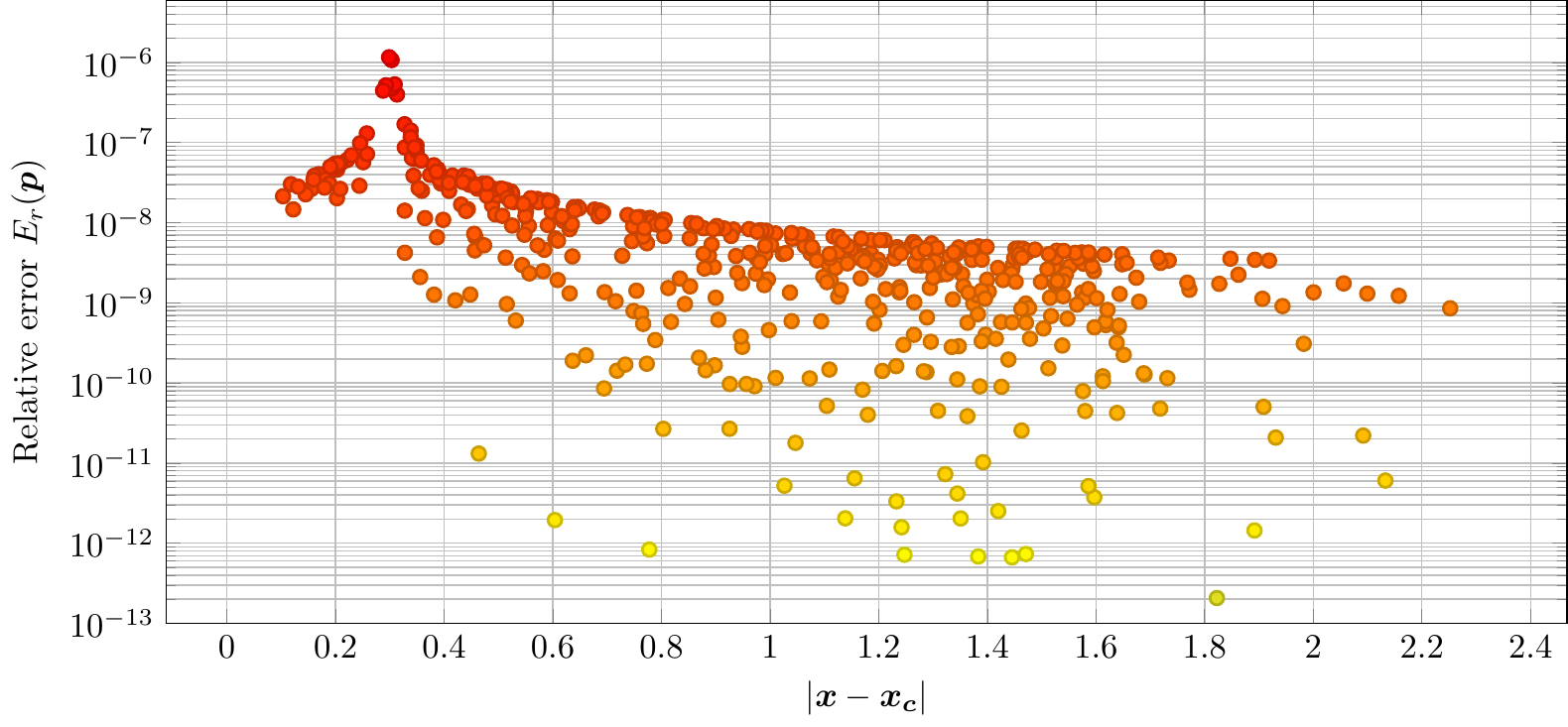}

    \caption{The relative error for the discrete level set describing a disk of radius $R=0.3$ as a function of the distance from the center $\boldsymbol{x_c} = (\frac{1}{2},\frac{1}{2})$.}
    \label{fig:Relative_error_discrete_level_set}
\end{figure}

\begin{algorithm}[H]
    \SetKwInOut{Input}{Input}
    \SetKwInOut{Output}{Output}
    \Input{$\gamma$ polygonal surface mesh, 
    \\ $\mathcal{I}_h$ coarse triangulation for the interior of $\gamma$,
    \\ $\boldsymbol{p} \in \Omega$.}
    \Output{$d(\boldsymbol{p},\gamma)$.}
    \eIf{$\boldsymbol{p} \in \mathcal{I}_h$}
      {
        $s \gets -1$
      }
      {
        $s \gets +1$
      }
    $\text{Find} \{\cellg_i\} \in \gamma_h \text{ nearest to } \boldsymbol{p}$. \\
    \For{$\cellg \in \{\cellg_i\}_i$}{
    $d_i \gets d(\boldsymbol{p},\cellg)$;
    }
    $\text{Return }s \cdot \min_i d_i$
    \caption{Evaluation of the discrete level set \eqref{eqn:discrete_signed_distance} for a given mesh $\gamma$.}
    \label{alg:discrete_level_set}
\end{algorithm}

\newpage
\section{Numerical experiments}
\label{sec:numberical_experiments}

Our implementation is based on the C++ finite
element library \texttt{deal.II} \cite{dealII94,dealIIdesign}, providing a
dimension independent user interface. The implementation of the Lagrange
multiplier and of the Nitsche's interface penalization methods are adapted from
the tutorial programs \texttt{step-60} and \texttt{step-70} of the deal.II
library, respectively, while the cut-FEM algorithm is adapted from the tutorial
program \texttt{step-85}, developed in~\cite{Sticko2016}.

As a result of this work, we added support and wrappers for the C++ library
\texttt{CGAL} (\cite{cgal:eb-22b},~\cite{CGALKernel}) into the \texttt{deal.II}
library~\cite{dealII94}, in order to perform most of the computational geometry
related tasks. Thanks to the so called \textit{exact computation paradigm}
provided by \texttt{CGAL}, which relies on computing with numbers of arbitrary
precision, our intersection routines are guaranteed to be robust.

We assume that the background mesh $\Omega_h$ is a $d$-dimensional triangulation and the immersed mesh $\gamma_h$ is $(d-1)$-dimensional %
with $d=2,3$. We validate our implementations with several experiments varying 
mesh configurations, algorithms, and boundary conditions. The source code used
to reproduce the numerical experiments is available from GitHub
\footnote{\href{https://github.com/fdrmrc/non_matching_test_suite.git}{https://github.com/fdrmrc/non\_matching\_test\_suite.git}}.

The tests are designed to analyze the performance of the methods presented in
Section~\ref{sec:methods} in different settings, varying the complexity of the
interface and the smoothness of the exact solution in both two and three
dimensions. All tests are performed using background meshes made of
quadrilaterals or hexahedra and immersed boundary meshes made of segments and
quadrilaterals. For the Lagrange multiplier and Nitsche's interface penalization
methods, we perform an initial pre-processing of the background grid $\Omega$ by
applying a localized refinement around the interface (where most of the error is concentrated), so that the resulting
number of degrees of freedom for the variable $u_h$ is roughly the same for all methods. 
Sample grids resulting from this process are shown in Figure~\ref{fig:interfaces}, where the interface has been resolved from
$\Omega_h$. We then proceed by computing errors and convergence rates against a
manufactured solution under simultaneous refinement of both the background and
immersed mesh. 

Classical $\mathcal{Q}^1$ Lagrangian elements are used for the background space
while piecewise constant elements are used to discretize the Lagrange
multiplier. For the Nitsche penalization method~\eqref{eqn:penalisation}, we
set the penalty parameter as $\beta=10$. Errors in the $H^1$- and $L^2$-norm  are reported for the main variable while for the Lagrange multiplier we use the discrete norm in \eqref{eqn:frac_norm} which, as already observed, is equivalent to the $H^{-1/2}(\gamma)$ norm on a quasi-uniform mesh. In the case of the Lagrange multiplier method, we
report the sum of the number of Degrees of Freedom (DoF) for $u_h$ and
$\lambda_h$ to underline the fact that a larger system must be solved, while
rates are computed against the number of DoF of each unknown.

For the (symmetric) penalized methods, the resulting linear systems are solved using a preconditioned conjugate gradient method, with an algebraic multigrid preconditioner based on the \texttt{Trilinos ML} implementation~\cite{Trilinos}, while for the Lagrange multiplier we exploit the preconditioner described
in Section~\ref{subsubsec:LM_discretisation} with the same preconditioned conjugate gradient method for inner solves of the stiffness matrices, and flexible GMRES as the outer solver. In the case of the Lagrange multiplier method, we list the total number of inner iterations required to invert the Schur complement.

\subsection{2D numerical tests}

In order to make the comparison between the three methods as close to real use cases as possible, we do not exploit any \emph{a-priori} knowledge of analytical level set descriptions of the exact interfaces. Indeed, using this information one could expect faster computations of the intersections, and an overall reduction of the computational costs of the assembly routines. Instead, we fix the same discretization of the interface as input data of the computational problem for all three methods.

In particular, we consider as computational domain $\Omega = [-1,1]^2$ with immersed domains of different shapes originating from an unfitted discretization of two different curves:
\begin{itemize}
    \item \emph{circle interface;} we let  $\gamma^1 \coloneqq \partial B_R(\boldsymbol{c})$, 
    \item \emph{flower-shaped interface;} we let $\gamma^2 \coloneqq \Bigl\{ (x,y): \sqrt{x^2+y^2} - r \Bigl(1 - 2 \frac{y^2}{x^2 + y^2} \Bigr) \Bigl(1 - 16 \frac{x^2y^2}{(x^2 + y^2)^2} \Bigr) - R =0 \Bigr\}$,
    
\end{itemize}
where the first is a circle of radius $R$ centered at $\boldsymbol{c}$ and the
second is a flower-like interface. We choose as parameter values
$\boldsymbol{c}=(\frac{1}{2},\frac{1}{2})$, $R=0.3$ and $r=0.1$. The initial
discretization of the immersed domains are chosen so that the geometrical error
is negligible w.r.t. the discretization errors (see
Figure~\ref{fig:interfaces}). In the case of the circle interface $\gamma^1$, the discrete interface is generated using a built-in mesh generatorsfrom
\texttt{deal.II} while the flower-shaped interface $\gamma^2$ is imported from an external file. When required, we implement a discrete level set function $\Psi_h$ as described in
Algorithm~\ref{alg:discrete_level_set}. Figure~\ref{fig:interfaces} shows a
representation of the two interfaces.

In the first two tests, we set up the problem using the method of manufactured
solutions, imposing the data $f$ in $\Omega$ and the boundary conditions on
$\Gamma$ and $\gamma$ according to the exact solution 
\begin{align}
u(x,y) \coloneqq \sin(2 \pi x) \sin(2 \pi y).
\end{align}
Hence, the right hand side of~\eqref{problem} is $f=8 \pi^2 u(x,y)$ and the data on the outer boundary $\Gamma$ and on $\gamma$ are computed accordingly. Notice that the smoothness of the solution implies $\lambda\equiv 0$ in this case. 

\begin{figure}[ht]
    \centering
    \subfloat[\centering]{{\includegraphics[width=7cm]{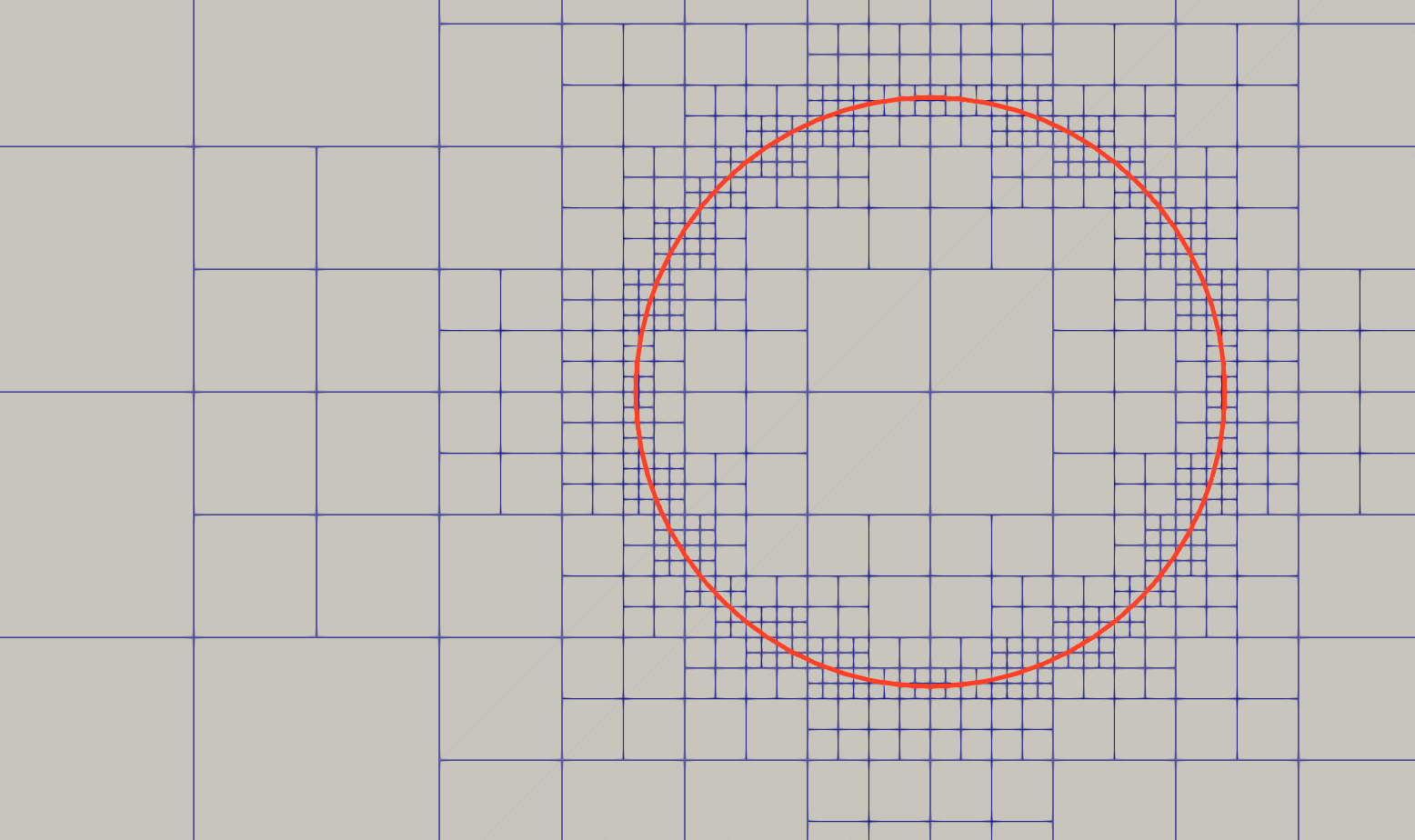}}}%
    \qquad
    \subfloat[\centering]{{\includegraphics[width=7cm]{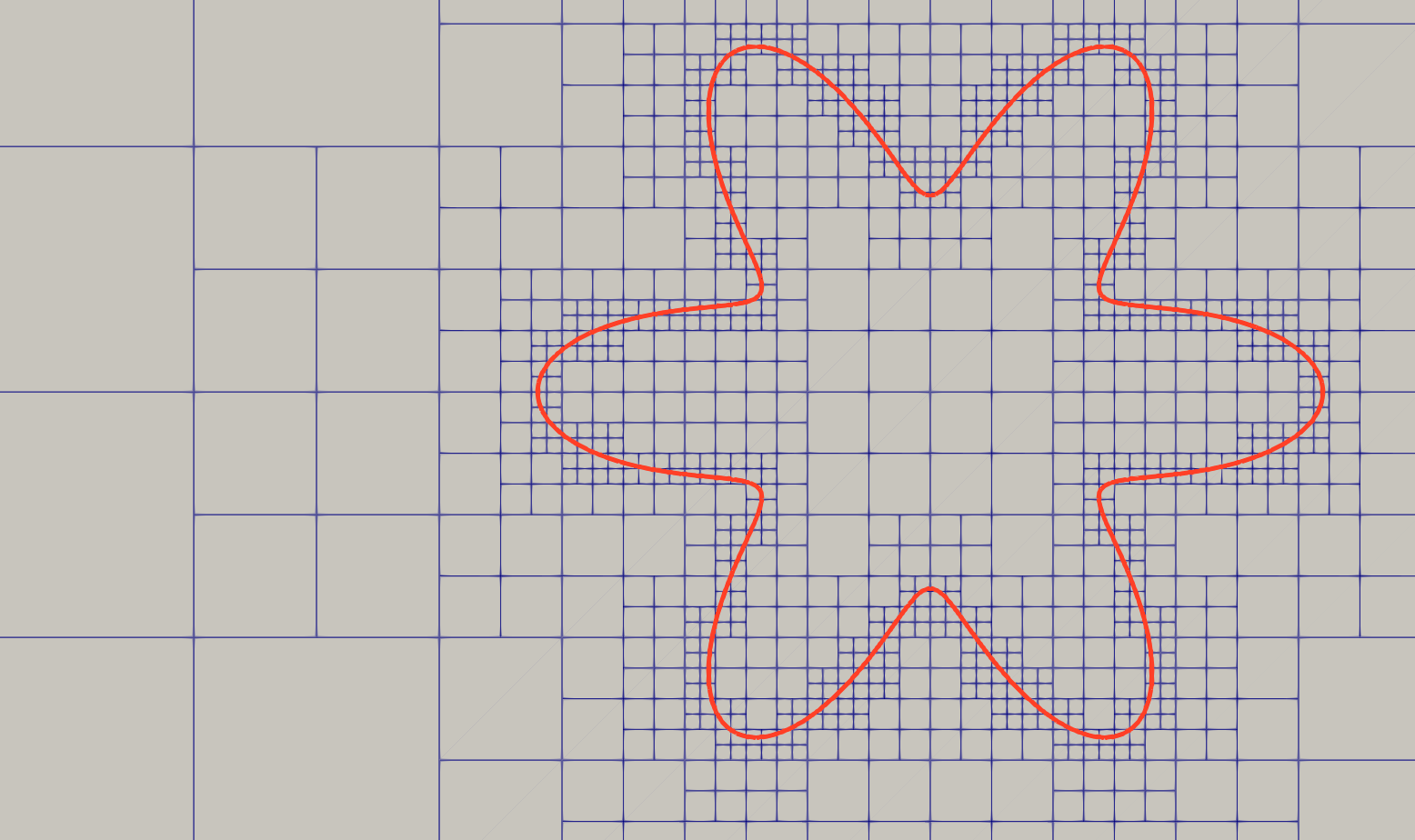} }}%
    \caption{Zoom on pre-processed background grid $\Omega_h$ for the circle interface $\gamma = \gamma^1$ (left) and the flower-shaped interface $\gamma = \gamma^2$ (right).}
    \label{fig:interfaces}
\end{figure}

This test is meant to assess the basic correctness of the implementation of the three methods, and corresponds to a case in which the interface is truly not an interface.%
Instead, when we impose an arbitrary value for the solution at the interface, we
expect the solution $u$ to be  only in $H^{\frac{3}{2}-\varepsilon}(\Omega)$ for
any $\varepsilon>0$, even thought its local regularity on the two subdomains
$\omega$ and $\Omega\setminus\omega$ may be higher. This is due to the fact that
the gradient of the solution is not a continuous function across the interface,
and therefore the solution cannot be in $H^2(\Omega)$. In this case, we cannot expect 
non-matching methods that does not resolve the interface \emph{exactly}, such as the
Lagrange multiplier and the Nitsche's interface penalization method, to be
able to recover the optimal rate of convergence.

In the tables below  we also report the number of iterations required in the solution phase in the column `Iter.'. We observe for all experiments a similar number of iterations for the three methods (which are independent on the number of degrees of freedom, indicating a good choice of preconditioner for all three methods) even though the solution of the linear system stemming from the Lagrange multiplier method is generally more expensive compared to the other two methods, owing to the higher computational complexity of the preconditoner for the saddle-point problem. The balance in the computational cost of the three different methods is discussed in details in
Section~\ref{sec:cpu-times}.

\subsubsection{\emph{Test 1: smooth solution over circular interface}}

We report in Tables \ref{tab:LM_gamma1}, \ref{tab:interface_penalisation_gamma1}, and \ref{tab:cut-FEM_gamma1} the errors and computed rates for the Lagrange multiplier, Nitsche's interface penalization, and cut-FEM method, respectively. 
In each case, the background variable converges linearly and quadratically in the $H^1$- and $L^2$-norm, respectively. As for the Lagrange multiplier, we observe a convergence rate close to two instead of the theoretical rate of one, most likely due to the very special exact solution that the multiplier converges to (i.e., the zero function).
For a direct comparison, we also report in Figure~\ref{fig:Convergence_smooth_2d} (left) the convergence history of all three methods against the number of DoF. These results clearly indicate that for smooth problems with relatively simple interfaces the three methods perform similarly.

\pgfplotstableread[comment chars={D}]{
    DOFs   L2          H1 
    329    8.7567e-02  2.1335e+00
    1161   2.2067e-02  1.0582e+00
    4377   5.4877e-03  5.2351e-01
   16953   1.3555e-03  2.5863e-01
   66665   3.3152e-04  1.2855e-01
}\gammaOneCutFEM
 
\pgfplotstableread[comment chars={D}]{
    DOFs   L2          H1         H-1/2
    389    5.579e-02   1.879e+00  7.402e-02
    1721   1.393e-02   9.391e-01  1.042e-02
    7217   3.479e-03   4.690e-01  2.805e-03
    29537  8.691e-04   2.343e-01  6.716e-04
    119489 2.172e-04   1.171e-01  2.078e-04
}\gammaOneLM

\pgfplotstableread[comment chars={D}]{
    Dofs   L2         H1 
     389 5.597e-02  1.879e+00
    1721 1.396e-02  9.391e-01
    7217 3.487e-03  4.690e-01
   29537 8.712e-04  2.343e-01
  119489 2.177e-04  1.171e-01
}\gammaOnePenalisation

\pgfplotstableread[comment chars={D}]{
    DOFs   L2          H1
      1209  2.2306e-02  1.0422e+00
      4487  5.7207e-03  5.2424e-01
     17163  1.4123e-03  2.5901e-01
     67089  3.4377e-04  1.2865e-01
    265249  8.5223e-05  6.4118e-02
}\gammaTwoCutFEM

\pgfplotstableread[comment chars={D}]{
    Dofs       H1       L2           H-1/2
      1958 5.327e-02   1.781e+00  6.551e-03
      8783 1.331e-02   8.916e-01  1.183e-03
     37037 3.326e-03   4.458e-01  5.237e-04
    151961  8.312e-04  2.228e-01  1.959e-04
    615473  2.078e-04  1.114e-01  5.327e-05
}\gammaTwoLM

\pgfplotstableread[comment chars={D}]{
    Dofs   L2         H1
   1958 5.327e-02  1.781e+00
   8783 1.331e-02  8.916e-01
  37037 3.326e-03  4.458e-01
 151961 8.312e-04  2.228e-01
 615473 2.078e-04  1.114e-01
}\gammaTwoPenalisation

\begin{table}[H]
    \begin{center}
        \begin{tabular}{ |p{2.2cm}||p{2cm}|p{1.1cm}|p{1.8cm}|p{1.1cm}|p{2cm}|p{1.1cm}|p{1cm}| }
            \hline
            \multicolumn{8}{|c|}{\textbf{Results for $\gamma = \gamma^1$ and smooth solution with Lagrange multiplier} } \\
            \hline
            DoF number & \textit{$\| u - u_h \|_{0,\Omega}$}& \textit{$L^2(\Omega)$ rate} & \textit{$\| u - u_h \|_{1,\Omega}$} & \textit{$H^1(\Omega)$ rate}  & \textit{$\| \lambda - \lambda_h \|_{-\frac{1}{2},\gamma}$} &  \textit{$H^{-\frac{1}{2}}(\gamma)$ rate} & Iter. \\
            \hline
               389+32 &   5.579e-02  &    -  &  1.879e+00 &     - &   7.402e-02 &    -  & 7  \\   
              1721+64  &  1.393e-02  & 1.87  &  9.391e-01 &  0.93 &   1.042e-02 & 2.83  & 7 \\ 
              7217+128  &  3.479e-03  & 1.94  &  4.690e-01 &  0.97 &   2.805e-03 & 1.89  & 9 \\ 
             29537+256  &  8.691e-04 & 1.97  &  2.343e-01 &  0.98 &   6.716e-04 & 2.06  & 11 \\ 
            119489+512 &  2.172e-04  & 1.98  &  1.171e-01 &  0.99 &   2.078e-04 & 1.69  & 11 \\ 
            \hline
        \end{tabular}
    \end{center}
    \caption{Rates in $L^2$ and $H^1$ for a smooth $u$ and $H^{-\frac{1}{2}}$ rates for the Lagrange multiplier method.}
    \label{tab:LM_gamma1}
\end{table}

\begin{table}[H]

\begin{center}
    \begin{tabular}{ |p{2.2cm}||p{2cm}|p{1.2cm}|p{2cm}|p{1.2cm}|p{1cm}| }
     \hline
     \multicolumn{6}{|c|}{\textbf{Results for $\gamma = \gamma^1$ and smooth solution with Nitsche} } \\
     \hline
     DoF number & \textit{$\| u - u_h \|_{0,\Omega}$}& \textit{$L^2(\Omega)$ rate} & \textit{$\| u - u_h \|_{1,\Omega}$} & \textit{$H^1(\Omega)$ rate} & Iter. \\
     \hline
        389 & 5.597e-02  &    -  &  1.879e+00   &     - & 1 \\
       1721 & 1.396e-02  & 2.13 &   9.391e-01   &  1.06 & 11\\
       7217 & 3.487e-03  & 2.06  &  4.690e-01   &  1.03 & 11 \\
      29537 & 8.712e-04  & 2.03  &  2.343e-01   &  1.02 & 12 \\
     119489 & 2.177e-04  & 2.02  &  1.171e-01   &  1.01 & 12 \\
     \hline  
    \end{tabular}
\end{center}
    \caption{$L^2$ and $H^1$ error rates for  $\gamma = \gamma^1$ and a smooth solution with Nitsche.}
    \label{tab:interface_penalisation_gamma1}
\end{table}

\begin{table}[H]

\begin{center}
    \begin{tabular}{ |p{2.5cm}||p{2cm}|p{1.2cm}|p{2cm}|p{1.2cm}|p{1cm}| }
     \hline
     \multicolumn{6}{|c|}{\textbf{Results for $\gamma = \gamma^1$ and smooth solution with cut-FEM} } \\
     \hline
     DoF number & \textit{$\| u - u_h \|_{0,\Omega}$}& \textit{$L^2(\Omega)$ rate} & \textit{$\| u - u_h \|_{1,\Omega}$} & \textit{$H^1(\Omega)$ rate} & Iter. \\
     \hline
        329 & 8.7560e-02 & -     &  2.1335e+00  &  -    & 2 \\
       1161 & 2.2064e-02 & 1.99  &  1.0582e+00  &  1.01 & 2 \\
       4377 & 5.4875e-03 & 2.01  &  5.2351e-01  &  1.02 & 15 \\
      16953 & 1.3554e-03 & 2.02  &  2.5863e-01  &  1.02 & 17 \\
      66665 & 3.3152e-04 & 2.03  &  1.2855e-01  &  1.01 & 18 \\
    
     \hline  
    \end{tabular}
\end{center}
    \caption{$L^2$ and $H^1$ error rates for  $\gamma = \gamma^1$ and a smooth solution with cut-FEM.}
    \label{tab:cut-FEM_gamma1}
\end{table}

\subsubsection{\emph{Test 2: smooth solution over flower-shaped interface}  }

We report in Tables \eqref{tab:LM_gamma2}, \eqref{tab:interface_penalisation_gamma2} and \eqref{tab:cut-FEM_gamma2} the error with the computed rates of convergence for the three schemes. Again we observe the theoretical rates of convergence for all three methods. This time, however, the direct comparison of the errors shown in Figure~\ref{fig:Convergence_smooth_2d} (right) indicates that the cut-FEM approach has an advantage over the other two methods. This may be due to the worst approximation properties of the non-matching methods based on Lagrange multipliers and Nitsche's interface penalization in the presence of high curvature sections of the immersed boundary, as can be seen by comparing the meshes shown in Figure~\ref{fig:interfaces}. Despite the fact that the reached accuracy is the same for the Lagrange multiplier and 
We further note that the Nitsche's interface penalization method requires less iterations than the Lagrange multiplier method. 

\begin{figure}[ht]
    \centering
    \includegraphics{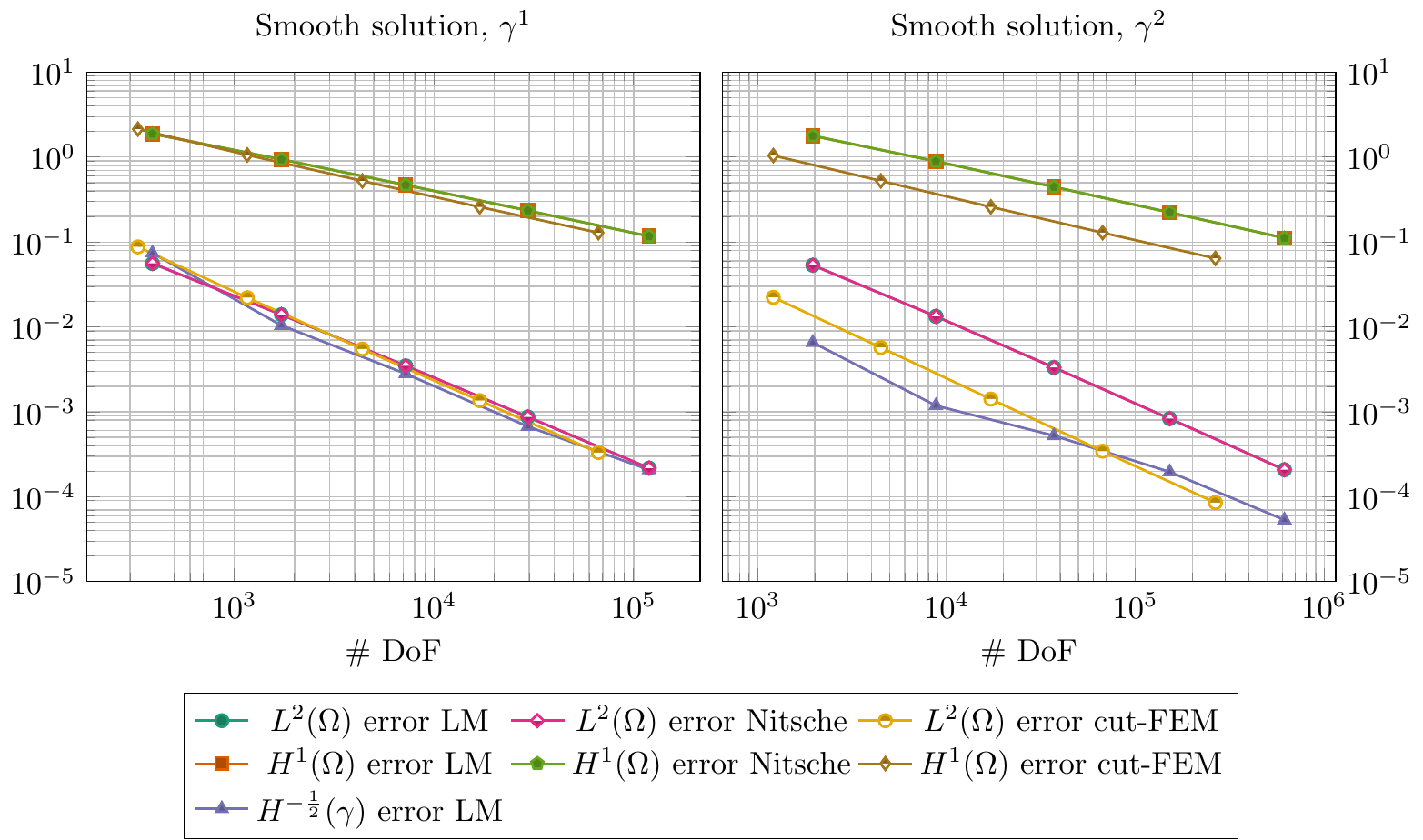}

   \caption{ $L^2$, $H^1$, and $H^{-\frac12}$ errors versus the number of DoF
   for all schemes applied to $\gamma^1$ (left) and to $\gamma^2$ (right).}%
    \label{fig:Convergence_smooth_2d}
\end{figure}

\begin{table}[H]
\begin{center}
\begin{tabular}{ |p{2.2cm}||p{2cm}|p{1.1cm}|p{1.8cm}|p{1.1cm}|p{2cm}|p{1.1cm}|p{1cm}| }
 \hline
 \multicolumn{8}{|c|}{\textbf{Results for $\gamma = \gamma^2$ and smooth solution with Lagrange multiplier} } \\
 \hline
 DoF number & \textit{$\| u - u_h \|_{0,\Omega}$}& \textit{$L^2(\Omega)$ rate} & \textit{$\| u - u_h \|_{1,\Omega}$} & \textit{$H^1(\Omega)$ rate} & \textit{$\| \lambda - \lambda_h \|_{-\frac{1}{2},\gamma}$} & \textit{$H^{-\frac{1}{2}}(\gamma)$ rate} & Iter. \\
 \hline
1958+256  & 5.327e-02  &    - &  1.781e+00  &    -  &   6.551e-03 &     - & 17 \\
8783+512  & 1.331e-02  & 1.85 &  8.916e-01  & 0.92  &   1.183e-03 &  2.47 & 20 \\
37037+1024 & 3.326e-03  & 1.93 &  4.458e-01  & 0.96  &  5.237e-04 &  1.18 &  33 \\
151961+2048 & 8.312e-04  & 1.96 &  2.228e-01  & 0.98  & 1.959e-04 &  1.42 &   49 \\
615473+4096 & 2.078e-04  & 1.98 &  1.114e-01  & 0.99  & 5.327e-05 &  1.88 &   53 \\
 \hline  
\end{tabular}
\end{center}
    \caption{$L^2$ and $H^1$ error rates for  $\gamma = \gamma^2$ and a smooth solution with Lagrange multiplier method.}
    \label{tab:LM_gamma2}
\end{table}

\begin{table}[H]

\begin{center}
    \begin{tabular}{ |p{2.2cm}|p{2cm}|p{1.1cm}|p{2cm}|p{1.1cm}|p{1.1cm}|p{1cm}| }
     \hline
     \multicolumn{6}{|c|}{\textbf{Results for $\gamma = \gamma^2$ and smooth solution with Nitsche} } \\
     \hline
     DoF number & \textit{$\| u - u_h \|_{0,\Omega}$}& \textit{$L^2(\Omega)$ rate} & \textit{$\| u - u_h \|_{1,\Omega}$} & \textit{$H^1(\Omega)$ rate} & Iter. \\
     \hline
      1958 & 5.327e-02 &    - &  1.781e+00 &     - & 13 \\
      8783 & 1.331e-02 & 1.85 &  8.916e-01 &  0.92 & 12 \\
     37037 & 3.326e-03 & 1.93 &  4.458e-01 &  0.96 & 13 \\
    151961 & 8.312e-04 & 1.96 &  2.228e-01 &  0.98 & 12 \\
    615473 & 2.078e-04 & 1.98 &  1.114e-01 &  0.99 & 14 \\
     \hline  
    \end{tabular}
\end{center}
    \caption{$L^2$ and $H^1$ error rates for  $\gamma = \gamma^2$ and a smooth solution with Nitsche}
    \label{tab:interface_penalisation_gamma2}
\end{table}

\begin{table}[H]

\begin{center}
    \begin{tabular}{ |p{2.2cm}|p{2cm}|p{1.1cm}|p{2cm}|p{1.1cm}|p{1.1cm}|p{1cm}| }
     \hline
     \multicolumn{6}{|c|}{\textbf{Results for $\gamma = \gamma^2$ and smooth solution with cut-FEM} } \\
     \hline
     DoF number & \textit{$\| u - u_h \|_{0,\Omega}$}& \textit{$L^2(\Omega)$ rate} & \textit{$\| u - u_h \|_{1,\Omega}$} & \textit{$H^1(\Omega)$ rate} & Iter. \\
     \hline
      1209 & 2.2306e-02 & -    &  1.0422e+00  &  -     & 2 \\
      4487 & 5.7207e-03 & 1.96 &  5.2424e-01  &  0.99  & 17 \\
     17163 & 1.4123e-03 & 2.02 &  2.5901e-01  &  1.02  & 17 \\
     67089 & 3.4377e-04 & 2.04 &  1.2865e-01  &  1.01  & 19 \\
    265249 & 8.5223e-05 & 2.01 &  6.4118e-02  &  1.00  & 20 \\
     \hline  
    \end{tabular}
\end{center}
    \caption{$L^2$ and $H^1$ error rates for  $\gamma = \gamma^2$ and a smooth solution with cut-FEM}
    \label{tab:cut-FEM_gamma2}
\end{table}

\subsubsection{\emph{Test 3: non-smooth solution over circular interface} }

\pgfplotstableread[comment chars={D}]{
    DOFs   L2          H1 
     389  1.413e-02  3.656e-01  3.354e-02
    1721  5.698e-03  2.254e-01  1.112e-02
    7217  3.248e-03  1.621e-01  6.273e-03
   29537  1.796e-03  1.141e-01  4.094e-03
  119489  1.099e-03  8.015e-02  2.845e-03
}\gammaOneNonSmoothLM

\pgfplotstableread[comment chars={D}]{
    DOFs   L2          H1 
      389  9.216e-03 3.667e-01
     1721  3.324e-03  2.286e-01
     7217  1.936e-03  1.640e-01
    29537  9.957e-04  1.151e-01
   119489  5.016e-04  8.070e-02
}\gammaOneNonSmoothPenalisation

\pgfplotstableread[comment chars={D}]{
    DOFs   L2          H1 
      329 3.1670e-02  5.6660e-01 
     1161 1.5998e-03  1.2352e-01 
     4377 3.6791e-04  5.8327e-02 
    16953 8.5282e-05  2.7307e-02 
    66665 2.1914e-05  1.3726e-02 
}\gammaOneNonSmoothCutFEM

In this test, we fix once more  $\gamma=\gamma^1$ and we define an exact solution with a non-zero jump of the normal gradient $\nabla u \cdot \boldsymbol{n}$ across $\gamma$ taken from \cite{Heltai2019}, namely
\begin{align}
\label{eqn:u_non_smmoth}
    u(x,y)& = 
    \begin{cases}
    -\ln(R)  \qquad \text{if } |r| \leq R \\
   -\ln(r) \qquad \text{  if } |r| > R, \\
    \end{cases}
\end{align}
where $r\coloneqq \boldsymbol{x} -\boldsymbol{c}$, implying as
right hand side $f=0$. The Lagrange multiplier associated to this solution is $\lambda(\boldsymbol{x}) \equiv \lambda =- \frac{1}{R}$, as $u$ solves the following classical interface problem:
\begin{align}
    \begin{cases}
        -\Delta u  &= 0  \qquad \text{ in } \Omega \setminus \gamma, \\
        u &= -\ln(r) \quad \text{  on } \Gamma, \\
       \llbracket \nabla u \cdot \boldsymbol{n} \rrbracket &= \frac{1}{R} \quad \text{ on } \gamma,  \\
        \llbracket u \rrbracket &=0 \quad \text{   on } \gamma. \\
    \end{cases}
\end{align}

Since the global regularity of the solution is $H^{\frac32-\varepsilon}(\Omega)$
for any $\varepsilon>0$, theoretically we would expect the convergence rates of
the $L^2(\Omega)$, $H^1(\Omega)$, and $H^{-\frac12}(\gamma)$ norms of the errors
to be $1$, $0.5$, and $0.5$ for the Lagrange multiplier method, the same for the
variable $u_h$ in the Nitsche's interface penalization method, and the optimal convergence
rates observed in the smooth case in the case of the cut-FEM method. These are
shown in Tables~\ref{tab:Non_smooth_LM},~\ref{tab:Non_smooth_penalisation}
and~\ref{tab:cut_1d2d_nonsmooth} and plotted in
Figure~\ref{fig:convergence-non-smooth}.

The results show a clear advantage in terms of convergence rates and absolute
values of the errors for the cut-FEM method. In all cases (both smooth and
non-smooth), the Nitsche's interface penalization method and the Lagrange multiplier give
essentially the same errors. A simple way to improve the situation for the latter
two methods would be to use a weighted norm during the computation of the error,
which was proven to be localized at the interface in \cite{Heltai2019}. This
would allow to reduce the overall error, but it would still result in a solution
that does not capture correctly the jump of the normal gradient across the
interface, which is the main source of error.

\begin{table}[H]

\begin{center}
    \begin{tabular}{ |p{2.2cm}||p{2cm}|p{1.1cm}|p{1.8cm}|p{1.1cm}|p{2cm}|p{1.1cm}|p{1cm}| }
     \hline
     \multicolumn{8}{|c|}{\textbf{Results for $\gamma = \gamma^1$ and non-smooth solution with Lagrange multipliers} } \\
     \hline
     DoF number & \textit{$\| u - u_h \|_{0,\Omega}$}& \textit{$L^2(\Omega)$ rate} & \textit{$\| u - u_h \|_{1,\Omega}$} & \textit{$H^1(\Omega)$ rate}  & \textit{$\| \lambda - \lambda_h \|_{-\frac{1}{2},\gamma}$} & \textit{$H^{-\frac{1}{2}}(\gamma)$ rate}  &Iter. \\
     \hline
       389+32 & 1.413e-02  &     - &  3.656e-01 &    - & 3.354e-02 &     - & 2 \\
      1721+64 & 5.706e-03  &  1.39 &  2.254e-01 & 0.74 & 1.112e-02 &  1.59 & 14 \\
      7217+128 & 3.250e-03  &  0.84 &  1.621e-01 & 0.49 & 6.273e-03 &  0.83 & 13 \\
     29537+256 & 1.796e-03  &  0.87 &  1.141e-01 & 0.51 & 4.094e-03 &  0.62 & 15 \\
    119489+512 & 1.099e-03  &  0.71 &  8.015e-02 & 0.51 & 2.845e-03 &  0.52 & 14 \\
    
     \hline
    \end{tabular}
\end{center}
    \caption{$L^2$-error and $H^1$-error for non smooth $u$ in \eqref{eqn:u_non_smmoth} and for the multiplier. }
    \label{tab:Non_smooth_LM}
\end{table}

\begin{table}[H]

\begin{center}
    \begin{tabular}{ |p{2.2cm}|p{2cm}|p{1.1cm}|p{2cm}|p{1.1cm}|p{1.1cm}|p{1cm}| }
     \hline
     \multicolumn{6}{|c|}{\textbf{Results for non-smooth $u$ with Nitsche} } \\
     \hline
     DoF number & \textit{$\| u - u_h \|_{0,\Omega}$}& \textit{$L^2(\Omega)$ rate} & \textit{$\| u - u_h \|_{1,\Omega}$} & \textit{$H^1(\Omega)$ rate} & Iter. \\
     \hline
       389 & 9.216e-03 &     - & 3.667e-01 &    - & 1 \\
      1721 & 3.324e-03 &  1.37 & 2.286e-01 & 0.64 & 11 \\
      7217 & 1.936e-03 &  0.75 & 1.640e-01 & 0.46 & 11 \\
     29537 & 9.957e-04 &  0.94 & 1.151e-01 & 0.50 & 12 \\
    119489 & 5.016e-04 &  0.98 & 8.070e-02 & 0.51 & 13  \\
     \hline
    \end{tabular}
\end{center}
    \caption{$L^2$-error, $H^1$-error for non smooth $u$ in \eqref{eqn:u_non_smmoth}.}
    \label{tab:Non_smooth_penalisation}
\end{table}

\begin{table}[H]
\begin{center}\begin{tabular}{ |p{2.2cm}|p{2cm}|p{1.1cm}|p{2cm}|p{1.1cm}|p{1.1cm}|p{1cm}| }
 \hline
 \multicolumn{6}{|c|}{\textbf{Results for non-smooth $u$ with cut-FEM } } \\
 \hline
  DoF number & \textit{$\| u - u_h \|_{0,\Omega}$}& \textit{$L^2(\Omega)$ rate} & \textit{$\| u - u_h \|_{1,\Omega}$} & \textit{$H^1(\Omega)$ rate} & Iter. \\
 \hline
   329 & 3.1670e-02 & -    & 5.6660e-01  &  -    & 2 \\
  1161 & 1.5998e-03 & 4.31 & 1.2352e-01  &  2.20 & 2   \\
  4377 & 3.6791e-04 & 2.12 & 5.8327e-02  &  1.08 & 16   \\
 16953 & 8.5282e-05 & 2.11 & 2.7307e-02  &  1.09 & 18   \\
 66665 & 2.1914e-05 & 1.96 & 1.3726e-02  &  0.99 & 19   \\
 \hline
\end{tabular}\end{center}
    \caption{$L^2$-error, $H^1$-error for non smooth $u$ in \eqref{eqn:u_non_smmoth} with cut-FEM.}
    \label{tab:cut_1d2d_nonsmooth}
\end{table}

\begin{figure}[ht]
    \centering
    \includegraphics{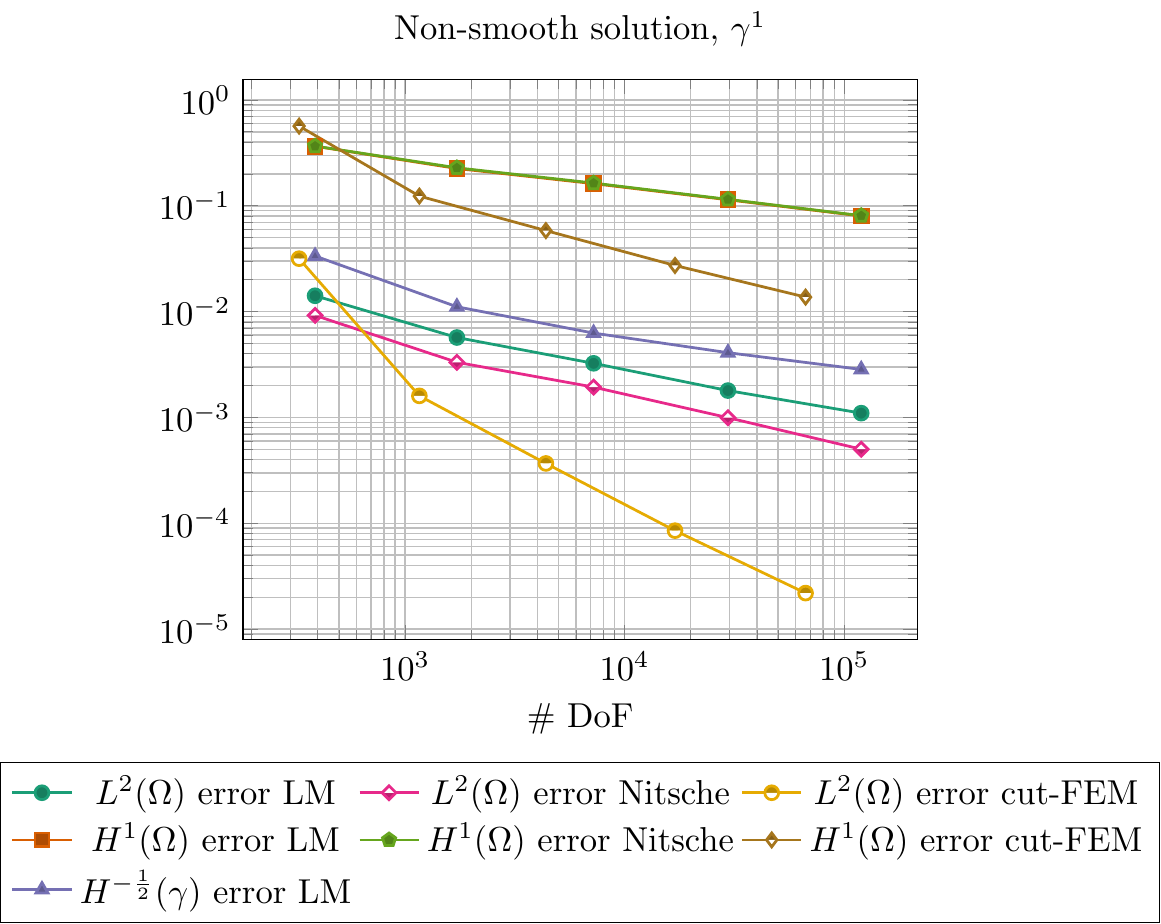}

   \caption{ $L^2$, $H^1$, and $H^{-\frac12}$ errors versus the number of DoF
   for all schemes applied to $\gamma^1$, with non-smooth solution.}%
    \label{fig:convergence-non-smooth}
\end{figure}

\subsection{3D numerical tests}

\pgfplotstableread[comment chars={D}]{
    DOFs   L2          H1 
   4127  6.416e-02  2.432e+00  0.1027 
  37031  1.590e-02  1.212e+00  0.0287 
 313007  3.963e-03  6.051e-01  0.0069
2572511  9.996e-04  3.024e-01  0.0019
}\gammaThreeLM

\pgfplotstableread[comment chars={D}]{
    DOFs   L2          H1 
  4127    6.411e-02  2.432e+00
  37031   1.588e-02  1.212e+00
  313007  3.959e-03  6.050e-01
  2572511 9.887e-04  3.023e-01
}\gammaThreePenalisation

\pgfplotstableread[comment chars={D}]{
    DOFs   L2          H1 
    5163 7.0847e-02 2.4835e+00  
   36781 1.7743e-02 1.2479e+00  
  278157 4.5124e-03 6.2396e-01  
2160541  1.1302e-03 3.1154e-01
}\gammaThreeCutFEM

\pgfplotstableread[comment chars={D}]{
    DOFs   L2          H1 
    4127      3.907e-02  1.251e+00   0.4552
    37031     2.041e-02  8.119e-01   0.2598
    313007    1.095e-02  5.764e-01   0.1806
    2572511   6.214e-03  4.253e-01   0.1216
}\gammaThreeNonSmoothLM

\pgfplotstableread[comment chars={D}]{
    DOFs   L2          H1 
    4127    8.688e-02 1.562e+00
    37031   2.157e-02 8.885e-01
    313007  4.893e-03 5.724e-01
    2572511 1.887e-03 3.954e-01
}\gammaThreeNonSmoothPenalisation

\pgfplotstableread[comment chars={D}]{
    DOFs   L2          H1 
   5163   6.3609e-02   1.3046e+00
  36781   6.9076e-03   4.5097e-01
 278157   1.5076e-03   2.2098e-01
 2160541  3.3802e-04   9.3590e-02
}\gammaThreeNonSmoothCutFEM

We mimic the tests reported above for the two-dimensional setting also in the
three-dimensional case, but we restrict our analysis to the case of a sphere
immersed in a box. We fix $\Omega = [-1,1]^3$ and consider as immersed interface
a sufficiently fine discretization of the sphere $\gamma^3 \coloneqq \partial
B_R(\boldsymbol{c})$ with
$\boldsymbol{c}=(\frac{1}{2},\frac{1}{2},\frac{1}{2})$, $R=0.3$.
In this setting, we consider two test cases with smooth and non-smooth solution.

\subsubsection{\emph{Test 1: smooth solution over spherical interface}}
\label{subsec:3Dsmooth}

We proceed analogously to the previous section by computing convergence rates when the solution $u$ is smooth and defined as 
\begin{align}
\label{eqn:u_smooth}
u(x,y,z) \coloneqq \sin(2 \pi x) \sin(2 \pi y) \sin(2 \pi z),
\end{align}
which corresponds to the right hand side $f=12 \pi^2 u(x,y,z)$. We report in Tables \ref{tab:LM_gamma1_3D}, \ref{tab:interface_penalisation_gamma1_3D} and \ref{tab:cut-FEM_gamma1_3D} the error rates for the three schemes, while in Figures \ref{fig:Convergence_sphere_3D} (left) we plot errors against the number of DoF.

The convergence rates are as expected and the results are in line with the
two-dimensional case. The Lagrange multiplier method and the interface
penalization method give again very close computational errors for $u_h$. As
in the smooth two-dimensional case, this test should only be considered as a
validation of the code and of the error computation, since the interface does
not have any effect on the computational solutions.

\begin{table}[H]
\centering
\begin{tabular}{ |p{2.2cm}||p{2cm}|p{1.1cm}|p{1.8cm}|p{1.1cm}|p{2cm}|p{1.1cm}|p{1cm}| }
 \hline
 \multicolumn{8}{|c|}{\textbf{Results for $\gamma = \gamma^3$ and smooth solution with Lagrange multiplier} } \\
 \hline
 DoF number & \textit{$\| u - u_h \|_{0,\Omega}$}& \textit{$L^2(\Omega)$ rate} & \textit{$\| u - u_h \|_{1,\Omega}$} & \textit{$H^1(\Omega)$ rate} & \textit{$\| \lambda - \lambda_h \|_{-\frac{1}{2},\gamma}$} & $H^{-\frac{1}{2}}(\gamma)$ rate & Iter. \\
 \hline
   4127+96 &   6.416e-02 &    - & 2.432e+00  &    - & 0.1027 &    -  &  9  \\
  37031+384 &  1.590e-02 & 1.91 & 1.212e+00  & 0.95 & 0.0287 & 1.84   & 18 \\
 313007+1536 & 3.963e-03 & 1.95 & 6.051e-01  & 0.98 & 0.0069 & 2.06  &  19 \\
2572511+6144 & 9.996e-04 & 1.96 & 3.024e-01  & 0.99 & 0.0019 & 1.89   & 14 \\
 \hline 
\end{tabular}
    \caption{$L^2$ and $H^1$ error rates for  $\gamma = \gamma^3$ and a smooth solution with Lagrange multiplier.}
    \label{tab:LM_gamma1_3D}
\end{table}

\begin{table}[H]
\begin{center}\begin{tabular}{ |p{2.2cm}|p{2cm}|p{1.1cm}|p{2cm}|p{1.1cm}|p{1.1cm}|p{1cm}| }
 \hline
 \multicolumn{6}{|c|}{\textbf{Results for $\gamma = \gamma^3$ and smooth solution with Nitsche} } \\
 \hline
 DoF number & \textit{$\| u - u_h \|_{0,\Omega}$}& \textit{$L^2(\Omega)$ rate} & \textit{$\| u - u_h \|_{1,\Omega}$} & \textit{$H^1(\Omega)$ rate} & Iter. \\
 \hline
   4127 & 6.411e-02 & -    &  2.432e+00 &  -    & 17 \\
  37031 & 1.588e-02 & 2.01 &  1.212e+00 &  1.00 & 14 \\
 313007 & 3.959e-03 & 2.00 &  6.050e-01 &  1.00 & 14 \\
2572511 & 9.887e-04 & 2.00 &  3.023e-01 &  1.00 & 14 \\
 \hline  
\end{tabular}\end{center}
    \caption{$L^2$ and $H^1$ error rates for  $\gamma = \gamma^3$ and a smooth solution with Nitsche.}
    \label{tab:interface_penalisation_gamma1_3D}
\end{table}

\begin{table}[H]
\begin{center}\begin{tabular}{ |p{2.2cm}|p{2cm}|p{1.1cm}|p{2cm}|p{1.1cm}|p{1.1cm}|p{1cm}| }
 \hline
 \multicolumn{6}{|c|}{\textbf{Results for $\gamma = \gamma^3$ and smooth solution with cut-FEM} } \\
 \hline
 DoF number & \textit{$\| u - u_h \|_{0,\Omega}$}& \textit{$L^2(\Omega)$ rate} & \textit{$\| u - u_h \|_{1,\Omega}$} & \textit{$H^1(\Omega)$ rate} & Iter. \\
 \hline
   5163 & 7.0847e-02  & -     &  2.4835e+00      &  - & 13 \\
  36781 & 1.7743e-02  & 2.12  & 1.2479e+00   &  1.05 & 12 \\
 278157 & 4.5124e-03  & 2.03  & 6.2396e-01   &  1.03 & 13 \\
2160541 & 1.1302e-03  & 2.03  & 3.1154e-01  &  1.02 & 13 \\
 \hline  
\end{tabular}\end{center}
    \caption{$L^2$ and $H^1$ error rates for  $\gamma = \gamma^3$ and a smooth solution with cut-FEM}
    \label{tab:cut-FEM_gamma1_3D}
\end{table}

\subsubsection{\emph{Test 2: non-smooth solution over spherical interface}}

We consider 
the test case in~\cite{Heltai2019}:
\begin{align}
    \label{eqn:u_non_smmoth_3D}
    u(x,y)& = 
    \begin{cases}
    \frac{1}{R}  &  \text{ if } |r| \leq R, \\
    \frac{1}{|r|} &   \text{if } |r| > R, \\
    \end{cases}
\end{align} where $r\coloneqq \boldsymbol{x} -\boldsymbol{c}$, $f=0$, and $R=0.3$.  Analogously to the two dimensional case, the multiplier associated to this problem is $\lambda(\boldsymbol{x}) = \lambda = -\frac{1}{R^2}$, since $u$ solves the following problem:
\begin{align}
    \begin{cases}
        -\Delta u  &= 0  \qquad \text{ in } \Omega \setminus \gamma, \\
        u &= \frac{1}{|r|} \quad \text{  on } \Gamma,\\
       \llbracket \nabla u \cdot \boldsymbol{n} \rrbracket &= \frac{1}{R^2} \quad \text{ on } \gamma,  \\
        \llbracket u \rrbracket &=0 \quad \text{   on } \gamma. \\
    \end{cases}
\end{align}
We report in Tables \ref{tab:LM_gamma1_3D_ns}, \ref{tab:interface_penalisation_gamma1_3D_ns}, \ref{tab:cutfem_non_smooth_sphere} the error rates for the three schemes and in Figure \ref{fig:Convergence_sphere_3D} (right) the error decay in $L^2(\Omega)$ and $H^1(\Omega)$ for $u_h$ and the decay in $H^{-\frac{1}{2}}(\gamma)$ for $\lambda_h$. A contour plot of the discrete solution $u_h$ obtained with Nitsche's penalization method is shown in Figure \ref{fig:contour_3D}. 

The difference between the three
methods  is less evident in terms of absolute values of the errors when compared
to the two dimensional case: while it is still clear that the convergence rates
of cut-FEM are higher compared to the other two methods, the difference between the
three methods is smaller. In particular, the Nitsche's interface penalization method
seems to perform better than the Lagrange multiplier method when considering the $L^2$ norm.

\begin{table}[H]
\centering
\begin{center}\begin{tabular}{ |p{2.2cm}||p{2cm}|p{1.1cm}|p{1.8cm}|p{1.1cm}|p{2cm}|p{1.1cm}|p{1cm}| }
 \hline
 \multicolumn{8}{|c|}{\textbf{Results for $\gamma = \gamma^3$ and non-smooth solution with Lagrange multiplier} } \\
 \hline
 DoF number & \textit{$\| u - u_h \|_{0,\Omega}$}& \textit{$L^2(\Omega)$ rate} & \textit{$\| u - u_h \|_{1,\Omega}$} & \textit{$H^1(\Omega)$ rate} & \textit{$\| \lambda - \lambda_h \|_{-\frac{1}{2},\gamma}$} & $H^{-\frac{1}{2}}(\gamma)$ rate & Iter. \\
 \hline
   4127+96&    3.907e-02 &    -  & 1.251e+00 &     -  &  0.4552 &    - & 2 \\
  37031+384&   2.041e-02 & 0.89  & 8.119e-01 &  0.59  &  0.2598 & 0.81 & 3  \\
 313007+1536 & 1.095e-02 & 0.88  & 5.764e-01 &  0.48  &  0.1806 & 0.52 & 3 \\
2572511+6446 & 6.214e-03 & 0.81  & 4.253e-01 &  0.43  &  0.1216 & 0.57 & 3 \\

 \hline 
\end{tabular}\end{center}
    \caption{$L^2$ and $H^1$ error rates for  $\gamma = \gamma^3$ and non-smooth solution $u$ in \ref{eqn:u_non_smmoth_3D} with Lagrange multiplier.}
    \label{tab:LM_gamma1_3D_ns}
\end{table}

\begin{table}[H]
\begin{center}\begin{tabular}{ |p{2.5cm}||p{2cm}|p{1.4cm}|p{2cm}|p{1.4cm}|p{1.5cm}| }
 \hline
 \multicolumn{6}{|c|}{\textbf{Results for $\gamma = \gamma^3$ and non-smooth solution with Nitsche} } \\
 \hline
 DoF number & \textit{$\| u - u_h \|_{0,\Omega}$}& \textit{$L^2(\Omega)$ rate} & \textit{$\| u - u_h \|_{1,\Omega}$} & \textit{$H^1(\Omega)$ rate} & Iter.\\
 \hline
   4127 &  8.688e-02 &    - & 1.562e+00 &     -  & 19   \\
  37031 &  2.157e-02 & 2.01 & 8.885e-01 &  0.81  & 16  \\
 313007 &  4.893e-03 & 2.14 & 5.724e-01 &  0.63  & 16  \\
2572511 &  1.887e-03 & 1.37 & 3.954e-01 &  0.53  & 16 \\
 \hline  
\end{tabular}\end{center}
    \caption{Rates for  $\gamma = \gamma^3$ and a non-smooth solution $u$ in \ref{eqn:u_non_smmoth_3D} with Nitsche. }
    \label{tab:interface_penalisation_gamma1_3D_ns}
\end{table}

\begin{table}[H]
\begin{center}\begin{tabular}{ |p{2.5cm}||p{2cm}|p{1.4cm}|p{2cm}|p{1.4cm}|p{1.5cm}| }
 \hline
 \multicolumn{6}{|c|}{\textbf{Results for $\gamma = \gamma^3$ and non-smooth solution with cut-FEM} } \\
 \hline
 DoF number & \textit{$\| u - u_h \|_{0,\Omega}$}& \textit{$L^2(\Omega)$ rate} & \textit{$\| u - u_h \|_{1,\Omega}$} & \textit{$H^1(\Omega)$ rate} & Iter.\\
 \hline
  5163 & 6.3609e-02 & -      & 1.3046e+00  & - &    25 \\
 36781 & 6.9076e-03 & 3.39  &  4.5097e-01 & 1.62 &  22 \\
278157 & 1.5076e-03 & 2.26  &  2.2098e-01 & 1.06 &  24 \\
2160541 & 3.3802e-04 & 2.19  & 9.3590e-02  & 1.26 &  24 \\
 \hline  
\end{tabular}\end{center}
    \caption{$L^2$ and $H^1$ error for  $\gamma = \gamma^3$ and a non-smooth solution $u$ in \ref{eqn:u_non_smmoth_3D} with cut-FEM}
    \label{tab:cutfem_non_smooth_sphere}
\end{table}

\begin{figure}[ht]
    \centering
    \includegraphics{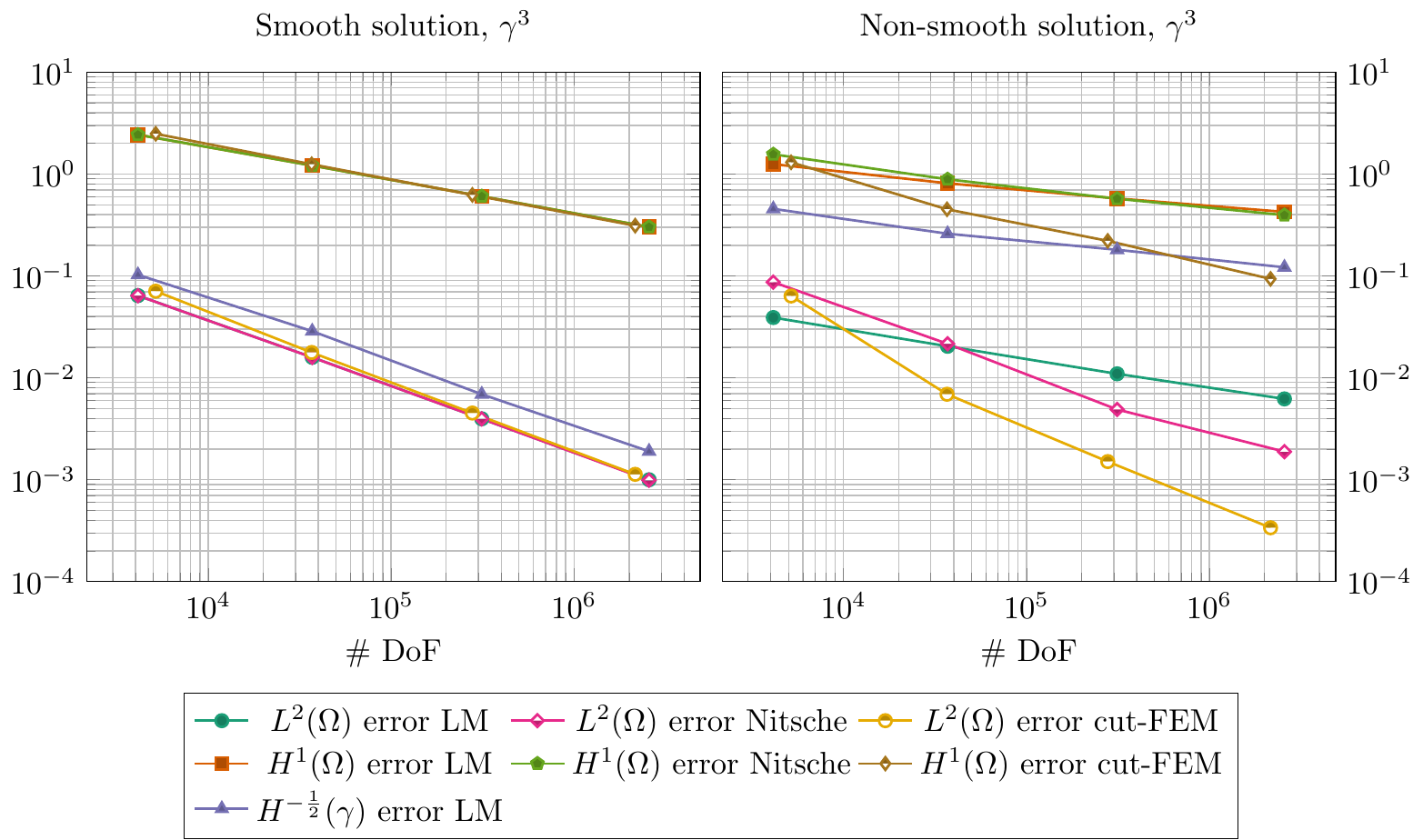}

    \caption{$L^2$, $H^1$, and $H^{-\frac12}$ error versus the number of DoF
    for all schemes applied to $\gamma^3$ with a smooth solution $u$ (left) and
    with a non-smooth solution $u$ (right).}
    \label{fig:Convergence_sphere_3D}
\end{figure}

\begin{figure}[ht]
    \centering
    \includegraphics[width=.49\textwidth]{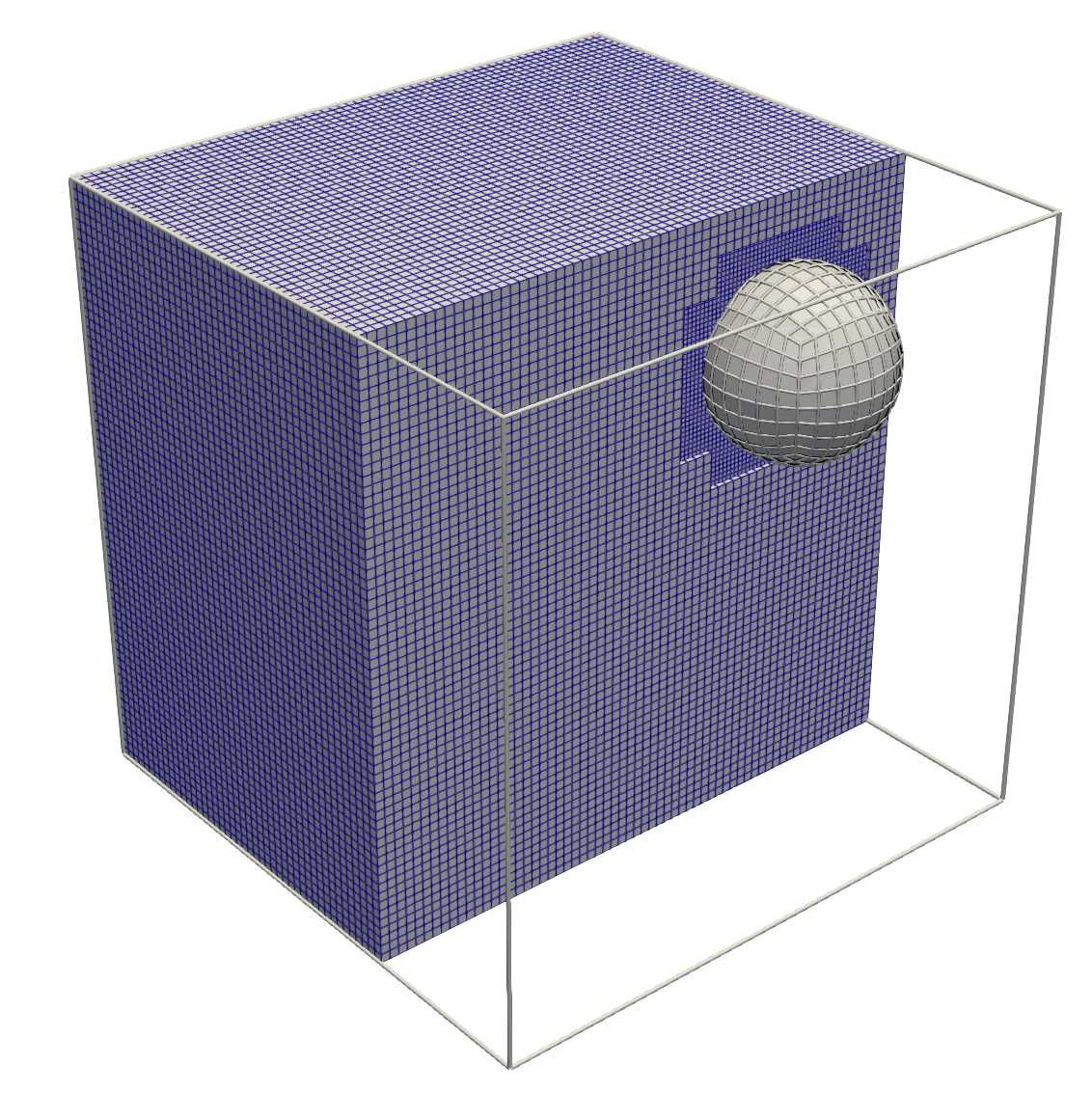}
    \hfill
    \includegraphics[width=.49\textwidth]{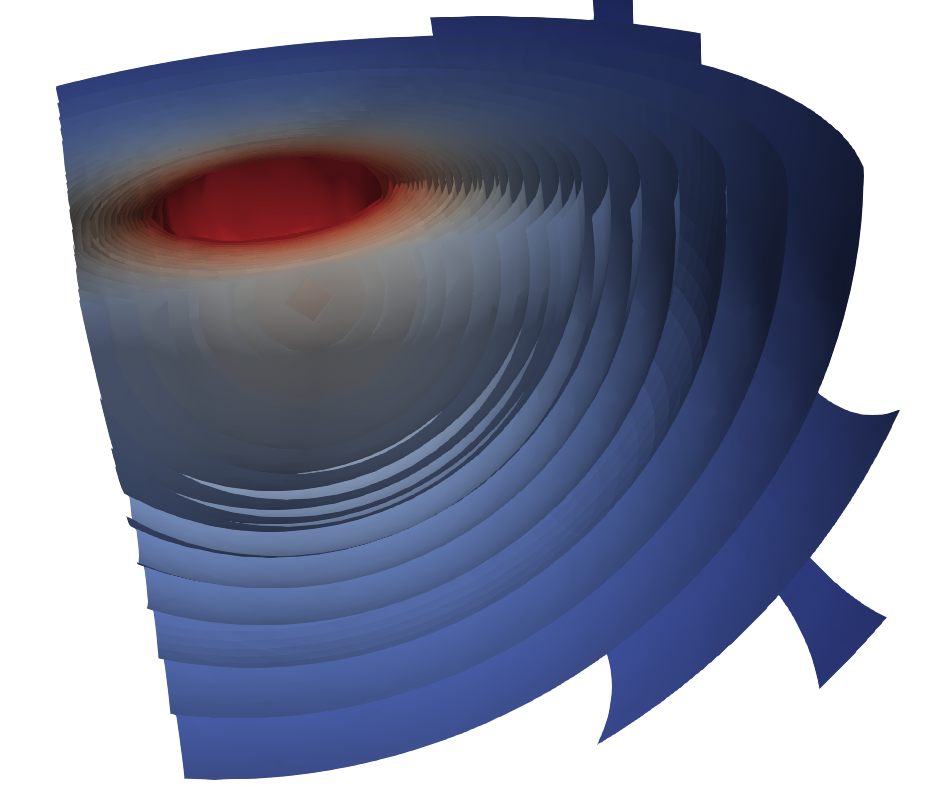}
    \caption{Background mesh $\Omega_h$ and immersed mesh of the sphere interface  $\gamma_h$ for the
    three dimensional case (left) and section of the contour plot for the
    approximate solution $u_h$ in \eqref{eqn:u_non_smmoth_3D},
    $\gamma=\gamma^3$.}
    \label{fig:contour_3D}
\end{figure}

\subsection{Computational times}
\label{sec:cpu-times}

In order to better understand the performance of the three methods, we consider a
breakdown of the computational costs into work precision diagrams. These provide a more fair measure of efficiency as they take into account the computational cost required to reach a given accuracy.

We report hereafter a breakdown of the computational times needed by our
implementations of the three proposed methods. All computations were carried out
on a 2.60GHz Intel Xeon processor. For each benchmark we report the average time required to solve 10 times the 3D smooth Problem~\ref{subsec:3Dsmooth}. 
We compute separately the required CPU times (in seconds) for the main tasks that each scheme has to perform. On a quasi-uniform mesh,
the number $N$ of background cells in $\Omega_h$ scales with
$\mathcal{O}(h_\Omega^{-3})$, and we expect the assembly of the stiffness matrix
to scale linearly in the number of cells. On the other hand, the number of
facets in $\gamma_h$ scales with $\mathcal{O}(h_\gamma^{-2})$; in our
experiments, the ratio $h_\Omega/h_\gamma$ is kept fixed, therefore we expect the
assembly of the coupling terms $\langle \lambda,v \rangle_{\gamma}$ and $\langle
u,v \rangle_{\gamma}$ to scale with $\mathcal{O}(N^{\frac{2}{3}})$.

This is indeed what we observe in the experiments as shown in CPU breakdown
plots of Figure~\ref{fig:computing_times_breakdown} for each method.

\begin{figure}[ht]
    \centering
    \includegraphics{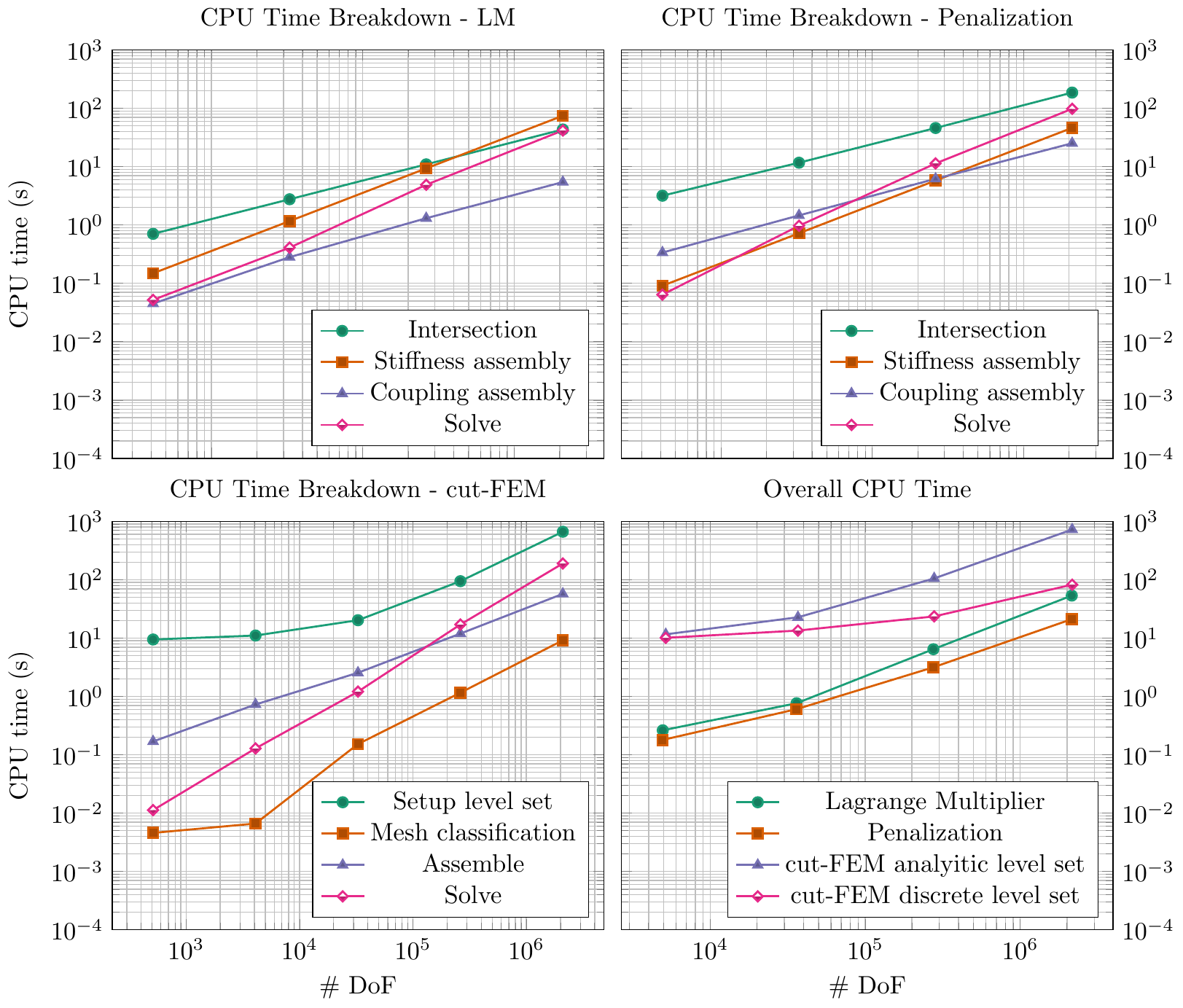}
    \caption{Breakdown of CPU times for the three schemes. In (bottom left) \emph{Setup level set} indicates the time required to interpolate the discrete level set described in Algorithm \ref{alg:discrete_level_set} onto the finite element space.
    \emph{Mesh classification} shows the time needed to partition the computational mesh and classify the cells in \emph{cut, interior} or \emph{outside} cells.
    Bottom right: overall CPU times to assemble the algebraic system for the smooth 3D test. For the cut-FEM method, we also show the CPU time obtained by using an analytical representation of the interface through an analytic level set function.
    }%
    \label{fig:computing_times_breakdown}
\end{figure}

The three schemes have comparable computational times. In particular, the Nitsche interface
penalization method exhibits lower global assembly times compared to the others.
However, it is well known that the number of iterations required to solve the
algebraic problem is influenced by the choice of the penalty parameter $\beta$
in \eqref{eqn:penalisation}, which determines simultaneously also the accuracy
of the numerical solution $u_h$. This can be better seen in work-precision
diagrams, where we compare the CPU times to solve each refinement cycle versus
the $L^2$ error for both test problems \eqref{eqn:u_smooth} and
\eqref{eqn:u_non_smmoth_3D} in Figure
\ref{fig:work_precision_diagrams}.
In the smooth scenario, results for the Lagrange multiplier and Nitsche's interface
penalization methods are almost overlapping both in terms of time and accuracy,
while cut-FEM shows larger computational times. The situation is different in the non-smooth case where cut-FEM better captures the discontinuity at the interface and thus gives  more accurate results, with a larger cost in terms of time for low degrees of freedom count, and with smaller cost for large degrees of freedom count, owing to the better convergence rate of the method. These results indicate that the additional implementation complexity
does pay back for non-smooth solutions.

\begin{figure}[ht]
    \centering
    \includegraphics{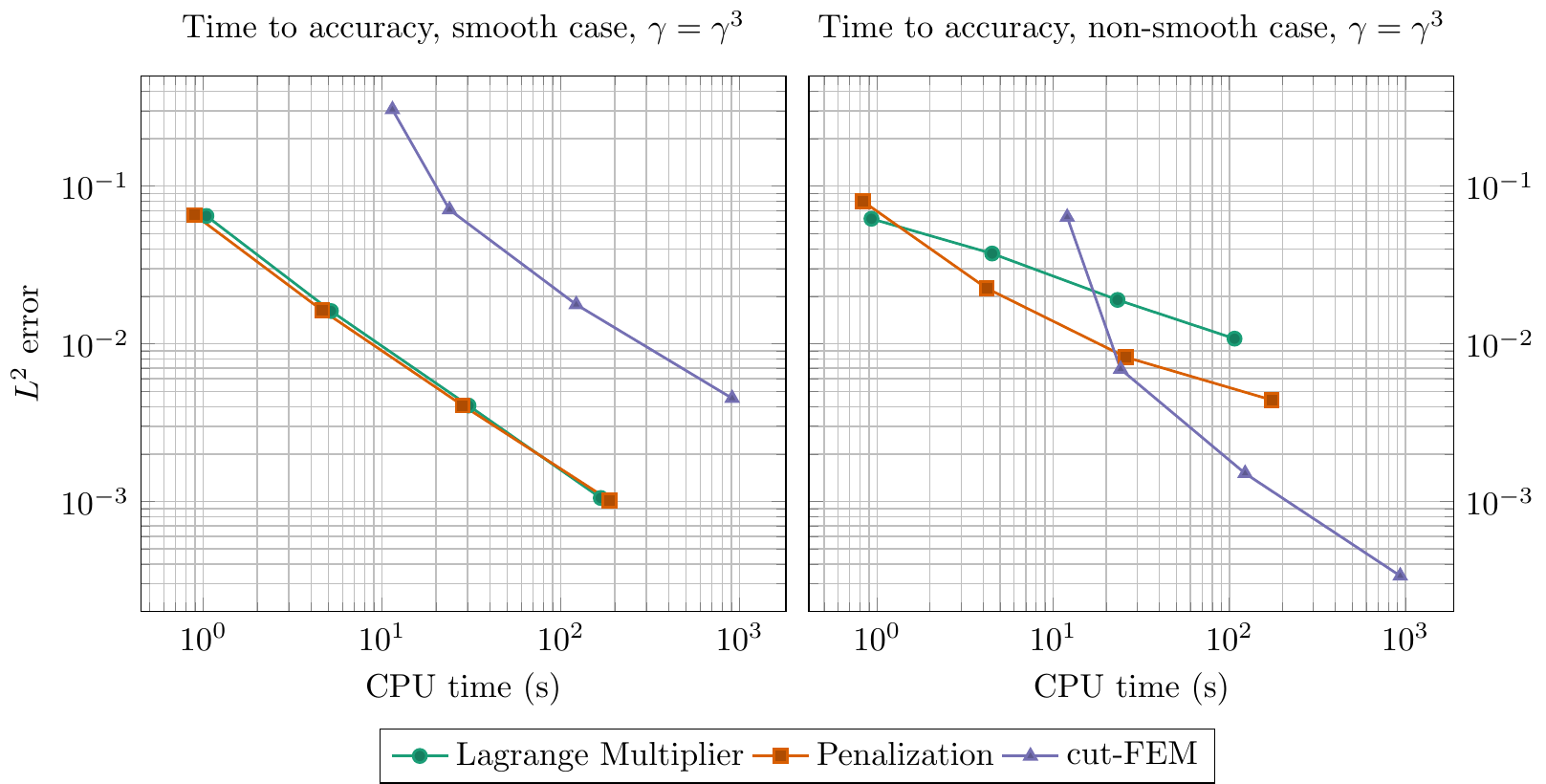}
  \caption{Work-precision diagrams for the two 3D tests.}%
  \label{fig:work_precision_diagrams}
\end{figure}

Based on the higher efficiency of the Lagrange multiplier and Nitsche penalization method in the smooth case, we speculate that these methods can be made competitive also in the non-smooth case by %
an improved local refinement strategy~\cite{HeltaiLei-2023-a}. 

\section{Conclusions}\label{sec:conc}

The numerical solution of partial differential equations modeling the interaction of physical phenomena across interfaces with complex, possibly moving, shapes is of great importance in many scientific fields. 
Boundary unfitted methods offer a valid alternative to remeshing and Arbitrary Lagrangian Eulerian formulations, but require care in the representation of the coupling terms between the interface and the bulk equations.

In this work, we performed a comparative analysis of three non-matching methods, namely the Lagrange multiplier method (or fictitious domain method), the Nitsche's interface penalization method, and the cut-FEM method, in terms of accuracy, computational cost, and implementation effort.

We presented the major algorithms used to integrate coupling terms on non-matching interfaces, discussed the benefit of computing accurate quadrature rules on mesh intersections, and concluded our analysis with a set of numerical experiments in two and three dimensions.

Our results show that accurate quadrature rules can significantly improve the accuracy of the numerical methods, and that there are cases in which simpler methods, like the Nitsche's interface penalization method, are competitive in terms of accuracy per computational effort. In general, the additional implementation burden of the cut-FEM method is justified by the higher accuracy that it achieves, even though the distance between cut-FEM and the other methods is not as large as one might expect, especially in three dimensions and for the solution of non-smooth problems.

We believe that this paper provides a valuable resource for researchers and
practitioners working on numerical methods for interface problems, particularly
in the context of elliptic PDEs coupled across heterogeneous dimensions. Source
codes used to produce the results of this paper are available at
\url{github.com/fdrmrc/non_matching_test_suite.git}. All the numerical experiments can be reproduced by using a \texttt{docker} running the shell
script \texttt{./scripts/run\_all\_tests.sh}.

\bibliography{refs}

\begin{thebibliography}{10}

\bibitem{dealIIdesign}
D.~Arndt, W.~Bangerth, D.~Davydov, T.~Heister, L.~Heltai, M.~Kronbichler,
  M.~Maier, J.-P. Pelteret, B.~Turcksin, and D.~Wells.
\newblock The deal.{II} finite element library: Design, features, and insights.
\newblock {\em Computers {\&} Mathematics with Applications}, 81:407--422, Jan.
  2021.

\bibitem{dealII94}
D.~Arndt, W.~Bangerth, M.~Feder, M.~Fehling, R.~Gassm{\"o}ller, T.~Heister,
  L.~Heltai, M.~Kronbichler, M.~Maier, P.~Munch, J.-P. Pelteret, S.~Sticko,
  B.~Turcksin, and D.~Wells.
\newblock The \texttt{deal.II} library, version 9.4.
\newblock {\em Journal of Numerical Math{ematics}}, 2022.

\bibitem{Babuska}
I.~Babuška.
\newblock The finite element method with {L}agrangian multipliers.
\newblock {\em Numerische Mathematik}, 20(3):179–192, 1973.

\bibitem{bbf}
D.~Boffi, F.~Brezzi, and M.~Fortin.
\newblock {\em Mixed finite element methods and applications}, volume~44 of
  {\em Springer Series in Computational Mathematics}.
\newblock Springer, Heidelberg, 2013.

\bibitem{Boffi_Credali_Gastaldi}
D.~Boffi, F.~Credali, and L.~Gastaldi.
\newblock On the interface matrix for fluid-structure interaction problems with
  fictitious domain approach.
\newblock {\em Comput. Methods Appl. Mech. Engrg.}, 401(part B):Paper No.
  115650, 23, 2022.

\bibitem{BG_numermath}
D.~Boffi and L.~Gastaldi.
\newblock A fictitious domain approach with {L}agrange multiplier for
  fluid-structure interactions.
\newblock {\em Numer. Math.}, 135(3):711--732, 2017.

\bibitem{Boffi_Gastaldi_Heltai_Peskin}
D.~Boffi, L.~Gastaldi, L.~Heltai, and C.~S. Peskin.
\newblock On the hyper-elastic formulation of the immersed boundary method.
\newblock {\em Computer Methods in Applied Mechanics and Engineering},
  197(25-28):2210--2231, 04 2008.

\bibitem{BoostLibrary}
Boost.
\newblock {Boost C++ Libraries}.
\newblock \url{http://www.boost.org/}, 2015.

\bibitem{BrennerScott}
S.~C. Brenner and L.~R. Scott.
\newblock {\em {The Mathematical Theory of Finite Element Methods}}, volume~15
  of {\em Texts in Applied Mathematics}.
\newblock Springer, 2008.

\bibitem{cutFEM}
E.~Burman, S.~Claus, P.~Hansbo, M.~G. Larson, and A.~Massing.
\newblock Cutfem: Discretizing geometry and partial differential equations.
\newblock {\em International Journal for Numerical Methods in Engineering},
  104(7):472--501, 2015.

\bibitem{Cangiani_PolyDG}
A.~Cangiani, E.~Georgoulis, and P.~Houston.
\newblock $hp$-version discontinuous {G}alerkin methods on polygonal and
  polyhedral meshes.
\newblock {\em Mathematical Models and Methods in Applied Sciences}, 24, 05
  2014.

\bibitem{Donea_ALE}
J.~{Donea}, S.~{Giuliani}, and J.~{Halleux}.
\newblock {An arbitrary lagrangian-eulerian finite element method for transient
  dynamic fluid-structure interactions}.
\newblock {\em Computer Methods in Applied Mechanics and Engineering},
  33(1-3):689--723, Sept. 1982.

\bibitem{CGALKernel}
A.~Fabri, G.-J. Giezeman, L.~Kettner, S.~Schirra, and S.~Sch{\"o}nherr.
\newblock The {CGAL} kernel: {A} basis for geometric computation.
\newblock {\em Lecture Notes in Computer Science}, pages 191--202, 1996.

\bibitem{GiraultGlowinski}
V.~Girault and R.~Glowinski.
\newblock Error analysis of a fictitious domain method applied to a {D}irichlet
  problem.
\newblock {\em Japan J. Indust. Appl. Math.}, 12(3):487--514, 1995.

\bibitem{GLOWINSKI1999755}
R.~Glowinski, T.-W. Pan, T.~Hesla, and D.~Joseph.
\newblock A distributed {L}agrange multiplier/fictitious domain method for
  particulate flows.
\newblock {\em International Journal of Multiphase Flow}, 25(5):755--794, 1999.

\bibitem{GlowinskiPanPeriaux1994}
R.~Glowinski, T.-W. Pan, and J.~P\'{e}riaux.
\newblock A fictitious domain method for {D}irichlet problem and applications.
\newblock {\em Comput. Methods Appl. Mech. Engrg.}, 111(3-4):283--303, 1994.

\bibitem{HeltaiLei-2023-a}
L.~Heltai and W.~Lei.
\newblock Adaptive finite element approximations for elliptic problems using
  regularized forcing data.
\newblock {\em {SIAM} Journal on Numerical Analysis}, 61(2):431--456, Mar.
  2023.

\bibitem{Heltai2019}
L.~Heltai and N.~Rotundo.
\newblock Error estimates in weighted sobolev norms for finite element immersed
  interface methods.
\newblock {\em Computers \& Mathematics with Applications}, 78(11):3586--3604,
  2019.

\bibitem{Trilinos}
M.~A. Heroux, R.~A. Bartlett, V.~E. Howle, R.~J. Hoekstra, J.~J. Hu, T.~G.
  Kolda, R.~B. Lehoucq, K.~R. Long, R.~P. Pawlowski, E.~T. Phipps, A.~G.
  Salinger, H.~K. Thornquist, R.~S. Tuminaro, J.~M. Willenbring, A.~Williams,
  and K.~S. Stanley.
\newblock An overview of the trilinos project.
\newblock {\em ACM Transactions on Mathematical Software}, 31(3):397--–423,
  sep 2005.

\bibitem{Hirt1997AnAL}
C.~W. Hirt, A.~A. Amsden, and J.~L. Cook.
\newblock An arbitrary {L}agrangian-{E}ulerian computing method for all flow
  speeds.
\newblock {\em Journal of Computational Physics}, 135:203--216, 1997.

\bibitem{parmoonolith}
R.~Krause and P.~Zulian.
\newblock A parallel approach to the variational transfer of discrete fields
  between arbitrarily distributed unstructured finite element meshes.
\newblock {\em SIAM Journal on Scientific Computing}, 38:C307--C333, 01 2016.

\bibitem{cut3D}
A.~Massing, M.~G. Larson, and A.~Logg.
\newblock Efficient implementation of finite element methods on nonmatching and
  overlapping meshes in three dimensions.
\newblock {\em SIAM Journal on Scientific Computing}, 35(1):C23--C47, 2013.

\bibitem{ExtendedFEM}
N.~Moës, J.~Dolbow, and T.~Belytschko.
\newblock A finite element method for crack growth without remeshing.
\newblock {\em International Journal for Numerical Methods in Engineering},
  46(1):131--150, 1999.

\bibitem{Nitsche}
J.~Nitsche.
\newblock \"{U}ber ein {V}ariationsprinzip zur {L}\"{o}sung von
  {D}irichlet-{P}roblemen bei {V}erwendung von {T}eilr\"{a}umen, die keinen
  {R}andbedingungen unterworfen sind.
\newblock {\em Abh. Math. Sem. Univ. Hamburg}, 36:9--15, 1971.

\bibitem{peskin_2002}
C.~S. Peskin.
\newblock The immersed boundary method.
\newblock {\em Acta Numerica}, 11:479–517, 2002.

\bibitem{Saye}
R.~I. Saye.
\newblock High-order quadrature methods for implicitly defined surfaces and
  volumes in hyperrectangles.
\newblock {\em SIAM Journal on Scientific Computing}, 37(2):A993--A1019, 2015.

\bibitem{Stenberg}
R.~Stenberg.
\newblock On some techniques for approximating boundary conditions in the
  finite element method.
\newblock {\em Journal of Computational and Applied Mathematics},
  63(1):139--148, 1995.
\newblock Proceedings of the International Symposium on Mathematical Modelling
  and Computational Methods Modelling 94.

\bibitem{Sticko2016}
S.~Sticko and G.~Kreiss.
\newblock A stabilized nitsche cut element method for the wave equation.
\newblock {\em Computer Methods in Applied Mechanics and Engineering},
  309:364--387, Sept. 2016.

\bibitem{cgal:eb-22b}
{The CGAL Project}.
\newblock {\em {CGAL} User and Reference Manual}.
\newblock {CGAL Editorial Board}, {5.5} edition, 2022.

\end{thebibliography}
\end{document}